\documentclass[a4paper, 11pt]{amsart}

\usepackage{graphicx}
\graphicspath{ {figures} }
\usepackage[]{amsmath}
\usepackage{overpic}
\usepackage{xcolor}
\usepackage{caption}
\usepackage{enumerate}
\usepackage{paralist}
\usepackage{tikz}
\usetikzlibrary{intersections, calc}

\newcommand{\Z}{{\mathbb Z}}
\newcommand{\R}{{\mathbb R}}
\newcommand{\C}{{\mathbb C}}

\newcommand{\sn}{\ensuremath\operatorname{sn}}
\newcommand{\cn}{\ensuremath\operatorname{cn}}
\newcommand{\dn}{\ensuremath\operatorname{dn}}

\newtheorem{theorem}{Theorem}[section]
\newtheorem{proposition}[theorem]{Proposition}
\newtheorem{lemma}[theorem]{Lemma}
\newtheorem{definition}[theorem]{Definition}

\newenvironment{remark}{\par\noindent{\em Remark.} 
\normalfont}{\par}
\newlength{\subfigureheight}
\makeatletter
\@fpsep\textheight
\makeatother

\title[Spherical and hyperbolic orthogonal ring patterns]{Spherical and hyperbolic orthogonal ring patterns: integrability and variational principles}

\author{Alexander I.~Bobenko}
\address{Institute of Mathematics, Secr. MA 8-3, TU Berlin, 10623 Berlin, Germany\vspace{-.5em}}
\address{Email: {\normalfont bobenko@math.tu-berlin.de}}

  
  \date{\today}
  
\begin{document}
\begin{abstract}
  We introduce orthogonal ring patterns in the 2-sphere and in the hyperbolic plane, consisting of pairs of concentric circles, which generalize circle patterns. We show that their radii are described by a discrete integrable system. This is a special case of the master integrable equation Q4. The variational description is given in terms of elliptic generalizations of the dilogarithm function. They have the same convexity principles as their circle-pattern counterparts. This allows us to prove existence and uniqueness results for the Dirichlet and Neumann boundary value problems. Some examples are computed numerically. In the limit of small smoothly varying rings, one obtains harmonic maps to the sphere and to the hyperbolic plane. A close relation to discrete surfaces with constant mean curvature is explained.
  \end{abstract}
\maketitle


\section{Introduction}
\label{sec:introduction}


The notion of discrete conformality is quite well developed. Circle patterns, discrete conformal maps, hyperbolic polyhedra, and their relationship are now a well-established and flourishing theory. Modern interest in this subject began with Thurston’s 
conjecture that circle packings could be used to approximate the classical Riemann map. Subsequently,  Rodin’s and Sullivan’s proof \cite{RoSu87} set off a flurry of research on discrete analytic functions and analytic maps based on packings and patterns of circles \cite{St05}. In particular, it was shown that on regular grids they are described by integrable systems \cite{Sch97}, and converge with all derivatives \cite{HeSch98}. Circle patterns on surfaces can be described variationally by convex functionals given by volumes of ideal hyperbolic polyhedra \cite{BSp04}. The corresponding theory of discrete minimal surfaces based on a discrete Weierstrass representation was developed in \cite{BHoSp06}.

In parallel, another approach to discrete conformal equivalence of triangulated piecewise Euclidean surfaces appeared: the discrete conformal metric change consists in multiplying all edge lengths by scale factors associated with the vertices \cite{Lu04}. Similar to the circle patterns, this theory also possesses convex variational principles \cite{SpSchPi08}, hyperbolic interpretation \cite{BPiSp15}, allowing investigation of global properties. It culminated in discrete uniformization theorems  \cite{GLSW18, GGLSW18, Sp20}. 

Conformal mappings are angle preserving, and are used for texture mappings in computer graphics because of this geometric property \cite{SpSchPi08, GiSpCr21, Cr20}. On the other hand, they are too rigid and do not allow local variations. Therefore, there has been a search for possible practical generalizations of conformal maps, see for example  \cite{CKPPS21}. In particular, there have been attempts to include anisotropy by using quasiconformal maps for this purpose. Since the theory of discrete quasiconformal maps is still missing, only numerical discretizations  have been applied \cite{WeMyZo12, LiKiFu09}.

There are numerous indications that the problems of structure-preserving discretizations of all these and other similar systems can be successfully studied in the framework of discrete differential geometry. For some of them, structured discrete models have recently been found. 
 By replacing the vertices of a polyhedral surface by circles, one obtains decorated discrete conformal maps \cite{BLu23, BLu24}. The classical theory then appears in the limit of infinitesimal circles.
The notion of decorated discrete conformal maps is closely related to inversive distance circle patterns \cite{BoSt04}, duality structures \cite{Gl11}, and unified discrete Ricci flow \cite{ZGZLYX14}. 
 Characteristic properties of all these theories that make them accessible for research are:
\begin{itemize}
\item a geometric description, often with a more transparent structure than in the smooth case,
\item a local analytic description in terms of geometric variables, which is usually an integrable system on a quad-graph, 
\item a variational (Lagrangian) description by a functional with similar convexity properties as in the smooth case,
\item  global existence and uniqueness results.
\end{itemize} 

Recently orthogonal ring patterns in the plane with the combinatorics of the square grid were introduced in \cite{BHoRo24}. The generalization from circles to rings introduces a new parameter into the theory, which is the area of a ring. Circle patterns appear then in a limit of ring patterns, when this parameter goes to zero. 

In this paper, we study orthogonal ring patterns in the sphere and in the hyperbolic plane. While the analytic descriptions of orthogonal ring and circle patterns in the plane are quite close, the situation with ring patterns in the sphere and in the hyperbolic plane turns out to be much more intriguing. We develop a theory of these ring patterns that has all the characteristic properties of an integrable geometric theory discussed above.

 A ring is a pair of two concentric circles that form an annulus. Two rings intersect orthogonally if the larger circle of one ring orthogonally intersects the smaller circle of the other and vice versa. In Section~\ref{sec:ring_patterns} based on these simple geometric properties we obtain a uniformization of the inner and outer radii of the rings by elliptic Jacobi functions and derive fundamental equations describing the ring patterns, (Theorem~\ref{thm:beta-equations}). Note that in contrast to the conformal character of orthogonal circle patterns, orthogonal ring patterns are not preserved by M\"obius transformations. Moreover, they are anisotropic: inner and outer circles of the rings touch in two different combinatorial directions.

The corresponding results for orthogonal ring patterns in the hyperbolic plane are presented in Section~\ref{sec:hyperbolic_ring_patterns}.  The circle pattern limit is described in Section~\ref{sec:cp_limit}.

The variational description plays a crucial role in the theory of circle patters \cite{CdV91, BSp04}. It is given, analytically, in terms of dilogarithm functions, and, geometrically, in terms of volumes of ideal hyperbolic polyhedra, the circle patters being closely related to \cite{Ri96}. 
A variational description of ring patterns  in terms of elliptic generalizations of the dilogarithm function is given in Section~\ref{sec:variational}. The functionals have the same convexity properties as their circle pattern counterparts. This allows us to study classical boundary value problems in Section~\ref{sec:computation}. In particular, we obtain the existence and uniqueness results (Theorem~\ref{thm:BVP_hyperbolic}) in the hyperbolic plane case using the convexity of the corresponding functional. In the spherical case the functional is not convex, which makes the treatment more complicated. For the computations we use a min-max principle, and prove the existence for ring patterns closed to the circle patterns, see Theorem~\ref{thm:circle-ring_deformation}.

In Section~\ref{sec:integrable} we show that spherical and hyperbolic orthogonal ring patterns are described by an integrable system, which is a special case of the Q4-equation \cite{Ad98}. The latter is an equation with an elliptic spectral parameter, and is the master discrete integrable equation in the classification of \cite{AdBSu03}. Finally, in Section~\ref{sec:smooth} we study the smooth limit of infinitesimal rings and obtain harmonic maps to the sphere and hyperbolic plane and the elliptic sinh-Gordon equations $\Delta u\pm \sinh u=0$ in the limit.   
 Note that finding an integrable discretization for these harmonic maps, as well as a geometric interpretation of the Q4-equation, have been long-standing problems in the theory of integrable systems.
 
Orthogonal ring patterns in the sphere and in the hyperbolic plane can be treated as the Gauss map of discrete cmc surfaces in $\R^3$ and $\R^{2,1}$. The corresponding theory was developed in a recent joint work with Tim Hoffmann and Nina Smeenk \cite{BHoSm24}. Discrete minimal surfaces of \cite{BHoSp06} appear here in the limit of circle patterns, see also \cite{BNeTe19}.

Learning from the history of the theory of circle patterns and discrete conformality, the following developments of the theory presented here seem to be quite promising. The first one is a generalization to ring patterns with arbitrary intersection angles. As a guiding principle here we are going to use the analytic description by the general Q4-equation. The second one is a generalization to ring patterns with arbitrary combinatorics, in particular, to triangulations. In both cases the goal is a theory that satisfies all the requirements of discrete differential geometry, with perfectly matching geometric and analytic descriptions. 

\subsection*{Acknowledgements} 
I am very grateful to Tim Hoffmann. This paper grew out of research on discrete surfaces with constant mean curvature, a project we have been working on together for years. The corresponding results are published in the recent companion paper \cite{BHoSm24}.  
I would like to thank Nina Smeenk for creating all the images used in this article and the accompanying software. Special thanks go to Yuri Suris for helpful discussions on the $Q4$ equation, to Thilo R\"orig for helpful discussions on ring pattern parametrizations and to Boris Springborn for helpful comments.

This research was supported by the DFG Collaborative Research Center TRR 109 ``Discretization in Geometry and Dynamics''.

\section{Orthogonal ring patterns in the sphere}
\label{sec:ring_patterns}

We consider patterns of rings on the sphere such that
neighboring rings intersect orthogonally, i.e., the larger circle of one ring
intersects the smaller ring of the other and vice versa.

Moreover these orthogonal ring patterns have the combinatorics of the square grid. Let $\Pi\subset \Z^2$ be a cell complex defined by a subset of elementary squares of the $\Z^2$ lattice in $\R^2$. We assume that $\Pi$ is simply connected, and identify it with the set of its vertices indexed by $(m,n)\in \Z^2$. The main example is a combinatorial rectangle
$$
\Pi=\{ (m,n)\in \Z^2 | 1\le m\le  M, 1\le n \le N \}. 
$$
We distinguish even and odd sublattices $\Pi=G \cup G^*$, denoting the corresponding subsets by
$$
G=\{ (m,n)\in \Pi | m+n=0 \ ({\rm mod} \ 2) \}, G^*=\{ (m,n)\in \Pi | m+n=1 \ ({\rm mod} \ 2) \}.
$$
They are dual to each other. $G$ will be combinatorially associated to the rings, and $G^*$ to their touching (intersection) points.

A \emph{spherical ring}  is a pair of two concentric circles on~$S^2$ that form a
ring (annulus). We identify the vertices of $G$ with the centers $k\in S^2$ of the rings and denote the
inner circle and its radius by small letters~$c$ and $r$, and the outer circle
and its radius by capital letters $C$ and $R$. Here the spherical radii $r$ and $R$ are not allowed to be larger then $\pi/2$. 
We assign an orientation to the
ring by allowing $r$ to be negative: positive radius corresponds to
counter-clockwise and negative radius to clockwise orientation.  The outer
radius will always be positive. Subscripts are used to associate circles and radii to vertices
of the complex, e.g., $c_{m,n}$ is the inner circle of the ring with the center $k_{m,n}$ associated with the
vertex~$v_{m,n}\in G$.

\begin{definition}[Orthogonal spherical ring patterns]
  \label{def:ring_pattern}
  Let $\Pi=G \cup G^*$ be a simply connected subcomplex of the $\Z^2$ lattice defined by its squares. 
  An \emph{orthogonal ring pattern} with the combinatorics of the square grid consists of rings $(C,c)$, 
  associated to even vertices $G$, and points $\ell\in S^2$,
  associated to odd vertices $G^*$ satisfying the following properties:
  \begin{enumerate}
    \item \label{def:rp_intersection}
      The rings associated to neighboring vertices $v_i$ and $v_j$
      \emph{intersect orthogonally}, i.e., the outer circle $C_i$ of the one
      vertex intersects the inner circle $c_j$ of the other vertex orthogonally
      and vice versa (see Fig.~\ref{fig:circle_ring_pattern}, left).
    \item \label{def:rp_touch}
      For any four consequently (neighboring) orthogonal rings associated to 
      $v_{m,n}, v_{m+1,n-1}, v_{m+2,n}, v_{m+1,n+1}$, (note $m+n=0 \ ({\rm mod} \ 2) $), 
      the inner circles $c_{m,n}$ and $c_{m+2,n}$ and the outer circles 
      $C_{m+1, n-1}$ and $C_{m+1,n+1}$ pass through one point $\ell_{m+1,n}$. 
      (Then orthogonality implies that the two inner and the two outer circles
      touch in this point. see Fig.~\ref{fig:circle_ring_pattern}, center).
    \item \label{def:rp_orientation}
    For any ring $(C_{m,n}, c_{m,n})$ the four touching points 
    $\ell_{m+1,n}=c_{m,n}\cap c_{m+2,n}$, $\ell_{m,n+1}=C_{m,n}\cap C_{m,n+2}$,  $\ell_{m-1,n}=c_{m,n}\cap c_{m-2,n}$ 
    and $\ell_{m,n-1}=C_{m,n}\cap C_{m,n-2}$ have the same orientation as~$c_{m,n}$, i.e., are in
      counter-clockwise order if $r_{m,n}$ is positive and in clockwise order if $r_{m,n}$ is negative. 
  \end{enumerate}
\end{definition}

We will call the vertices $v_{m\pm 1, n\pm 1}$ \emph{neighboring} to the vertex $v_{m,n}$, and the vertices $v_{m,n\pm 2}, v_{m\pm 2, n}$ \emph{next-neighboring} to $v_{m,n}$. 
\begin{remark}
Let us note that in contrast to conformal character of circle patterns, orthogonal ring patterns are \emph{anisotropic}. The next-neighboring rings touch by inner circles in the direction of the (horizontal) $m$-axis, and by outer circles in the direction of the (vertical) $n$-axis, see Fig.~\ref{fig:circle_ring_pattern}, center. So the distances between the ring centers are smaller in the horizontal direction, than in the vertical direction. 
\end{remark}

\begin{figure}[t]
 \centering
  \includegraphics[width=.9\linewidth]{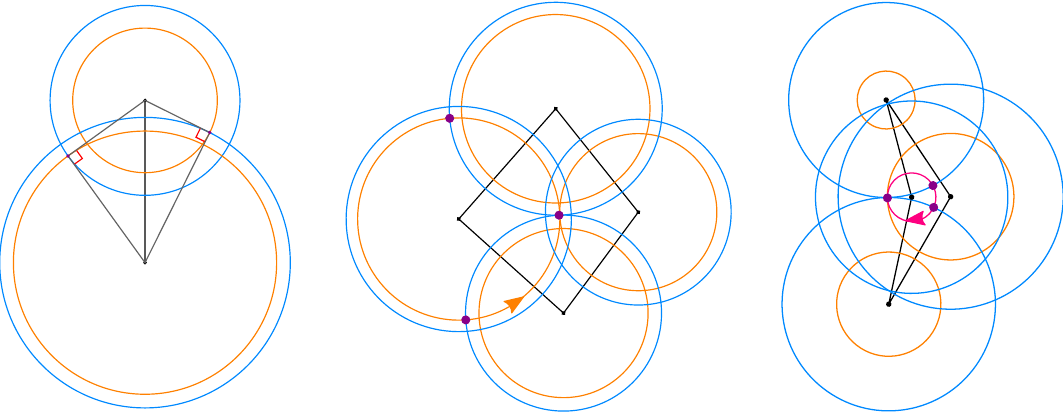}
  \caption{Left: Two orthogonally intersecting rings.
    Center: The inner circles touch along one diagonal of a quadrilateral and the
    outer circles along the other diagonal. The touching point coincides.
    Right: If the orientation (i.e., signed radii) of the inner circles differ,
    then the centers lie on the same side of the common tangent.
  }
  \label{fig:circle_ring_pattern}
\end{figure}

Applying the spherical Pythagoras' Theorem to two orthogonally intersecting rings we obtain: 
\begin{align} 
  \label{eq:q}
  \cos R \cos r_k = \cos r \cos R_k,   
\end{align}
where $r$ and $R$ are spherical radii of  the smaller and larger circles respectively. As an implication we obtain 

\begin{lemma}
There exists a global constant $q\le 1$ such that the radii of inner and outer circles of all rings of an orthogonal ring pattern are related by 
\begin{equation}
\label{eq:q-spherical}
q  \cos r  =  \cos R.
\end{equation}  
\end{lemma}

\begin{remark}
In the limit of small circles
$$
R=\epsilon \tilde{R}, r=\epsilon\tilde{r}, q=(1-\epsilon^2\delta+o(\epsilon^2))
$$
we obtain orthogonal ring patterns in the euclidean plane introduced in \cite{BHoRo24}. Identity (\ref{eq:q-spherical}) implies that they have constant area $\pi(\tilde{R}^2-\tilde{r}^2)=2\pi\delta$.
\end{remark}

We parametrize spherical rings satisfying (\ref{eq:q-spherical}) using Jacobi elliptic functions of modulus $q$. These are doubly periodic functions on the rectangular torus with the lattice given by the real  $4K$ and imaginary $4iK'$ periods. 

\begin{proposition}
\label{prop:sphere_ring_uniformization}
The formulas 
\begin{equation}
\label{eq:sphere_ring_uniformization}
\cos r=\sn (\beta,q) ,\ \sin r=\cn (\beta,q), \ \sin R=\dn (\beta, q)
\end{equation}
describe a one to one correspondence between the radii $(r,R)\in [-\pi, \pi] \times [R_0,\pi-R_0]$ of a spherical ring with the parameter $q\le 1$ given by (\ref{eq:q-spherical}) and the variable $\beta\in [0, 4K] $ of Jacobi functions of modulus $q$. Here $R_0=\arccos q$.
\end{proposition}

\begin{proof}
The first two equations relate $r$ and $\beta$, we note that $r$ may be negative. 
Then the last formula in (\ref{eq:sphere_ring_uniformization}) determines $R$ up to the ambiguity $R\leftrightarrow \pi -R$.
The identity $\dn^2\beta+q^2\sn^2\beta=1$ is equivalent to $q^2\cos^2 r=\cos^2 R$. Finally the identity (\ref{eq:q-spherical}) fixes the sign of $\cos R$ and thus the choice of $R$. 
\end{proof}

The correspondence of the intervals of this parametrization is shown in the Table.
\begin{center}
\begin{tabular}{l c r}
  \label{eq:spherical_angles}
$\beta$  & $ r $ & $R$ \\ 
 $ [0,K] $ &$ [0,\frac{\pi}{2}]$ & $[R_0,\frac{\pi}{2}]$\\
  $ [K,2K] $ &$ [-\frac{\pi}{2},0]$ & $[R_0,\frac{\pi}{2}]$\\ 
  $ [2K,3K] $ &$ [-\pi,-\frac{\pi}{2}]$ & $[\frac{\pi}{2}, \pi -R_0]$\\ 
  $ [3K,4K] $ &$ [\frac{\pi}{2},\pi]$ & $[\frac{\pi}{2}, \pi -R_0]$
\end{tabular}
\end{center}

\begin{figure}[t]
	\centering
	\begin{overpic}[scale=.2]
		{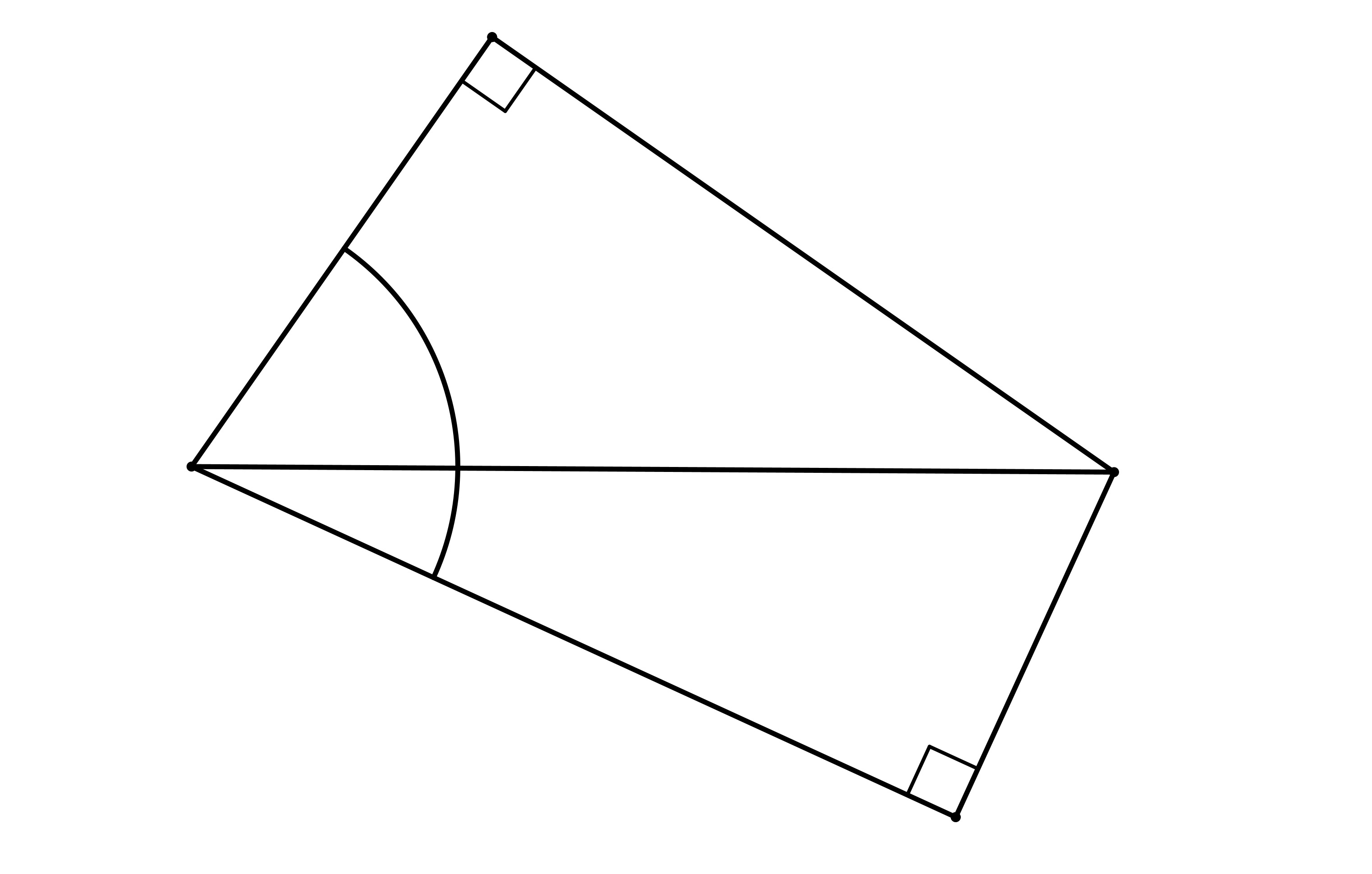}
		\put(22,34){$\varphi_k$}
		\put(24,25){$\psi_k$}
		
		\put(8,28){$v$}
		\put(83,28){$v_k$}
		
		\put(23,48){$r$}
		\put(35,10){$R$}
		
		\put(57,48){$R_k$}
		\put(77,14){$r_k$}
	\end{overpic}
	\caption{An orthogonal quadrilateral spanned by two rings.}
	\label{fig:kite}
\end{figure}
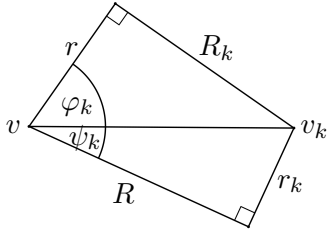
Consider an edge $(v, v_k)$ in the pattern. 
The angles in an orthogonal quadrilateral spanned by two rings with radii $(r,R)$ and $(r_k, R_k)$, see Fig.~\ref{fig:kite}, can be computed using Napier's rule:
\begin{align*}
 \varphi_k = \arctan \frac{\tan R_k}{\sin r}
 \quad \text{and} \quad   
\psi_k  = \arctan \frac{\tan r_k}{\sin R}.
\end{align*}

In terms of the elliptic functions parametrization the angles are given by
\begin{align}
  \label{eq:spherical_intervals}
  \varphi_k &= 
  \arctan\left( 
    \frac{\dn \beta_k}{q \sn \beta_k \cn \beta}
  \right)
	    &\text{and}&&
  \psi_k &=
  \arctan\left( 
  \frac{\cn \beta_k}{\sn \beta_k \dn \beta}
  \right).
\end{align}

They can be transformed to the following form
\begin{eqnarray*}
  \varphi_k &=& \arg \left(1+i \frac{\dn \beta_k}{q \sn \beta_k \cn \beta}\right) = \arg \left(1- \frac{\cn (\beta_k+iK')}{\cn \beta}\right),\\
  \psi_k &= &\arg \left(1+i\frac{\cn \beta_k}{\sn \beta_k \dn \beta}\right)=\arg \left(1-\frac{\dn (\beta_k+iK')}{\dn \beta}\right). 
\end{eqnarray*}

Note that $ \cn (\beta_k+iK')$ and $\dn (\beta_k+iK')$ are pure imaginary. For the sum this implies
$$
\theta_k:=\varphi_k+\psi_k= \arg \left(  \frac{1}{\cn\beta \dn\beta} \frac{\cn\beta -\cn (\beta_k+iK') }{\dn\beta +\dn (\beta_k+iK')}\right).
$$

We compute a product representation using the following addition formulas for
Jacobi elliptic functions
\begin{eqnarray}
\label{eq:addition_formulas}
  \cn(x+y) - \cn(x-y) &= -
  \frac
  {2 \sn (x) \sn (y) \dn (x) \dn (y)} 
  {1 - q^2 \sn^2 (x) \sn^2 (y)}
  \,, \\
  \dn(x+y) + \dn(x-y) &= 
  \frac
  {2 \dn (x) \dn (y)} 
  {1 - q^2 \sn^2 (x) \sn^2 (y)}. 
  \nonumber
\end{eqnarray}
So with $x = \frac{\beta + \beta_k + i K'}{2}$ and $y = \frac{\beta - \beta_k - i K'}{2}$ we obtain:
\begin{eqnarray*}
 \theta_k = &\arg\left(-{\rm sign}(r) \sn\frac{\beta+\beta_k+ iK'}{2}\sn\frac{\beta-\beta_k- iK'}{2}\right)= \\
 &\arg\left(-{\rm sign}(r) \frac{\sn\frac{\beta+\beta_k+ iK'}{2}}{\sn\frac{\beta-\beta_k+ iK'}{2}}\right).
\end {eqnarray*}
Here we used that $\dn\beta >0$ and ${\rm sign}(\cn\beta)={\rm sign}(r)$.
 
Finally the sum of the angles is given by
\begin{eqnarray}
  \theta_k &=&\pi + \arg\left(\frac{\sn\frac{\beta+\beta_k+ iK'}{2}}{\sn\frac{\beta-\beta_k+ iK'}{2}}\right), 
  \quad \text{if}\  r>0  \Leftrightarrow  \beta\in [-K,K], 
  \label{eq:angle_positive_r}\\
  \theta_k &=&\arg\left(\frac{\sn\frac{\beta+\beta_k+ iK'}{2}}{\sn\frac{\beta-\beta_k+ iK'}{2}}\right), 
  \quad \text{if}\ r<0  \Leftrightarrow  \beta\in [K,3K].
   \label{eq:angle_negative_r}
  \end{eqnarray}
 Note that 
$$
\sn(\beta-\frac{iK'}{2})=\frac{1}{q\sn(\beta+\frac{iK'}{2})},\quad  |\sn(\beta+\frac{iK'}{2})|=\frac{1}{\sqrt{q}}.
$$

An internal ring of an orthogonal ring pattern has four neighboring rings, which consequently touch. We call this configuration of five rings {\em a ring flower}. 
The existence of a ring flower for any internal ring is equivalent to conditions (1) and (2) of Definition ~\ref{def:ring_pattern}. We define {\em generalized orthogonal ring patters} as sets of rings with the combinatorics of the square grid, such that every internal ring possesses a ring flower. Generalized ring patterns do not necessarily satisfy condition (3) of Definition ~\ref{def:ring_pattern}.

We arrive at the following conclusion.
\begin{theorem} 
\label{thm:spherical}
Rings build a generalized orthogonal ring pattern in a sphere if and only if they are given by the variables $\beta$ satisfying
  \begin{equation}  
  \label{eq:spherical}
    \prod_{v_k \sim v} 
    \dfrac
    { \sn (\dfrac{\beta - \beta_k + i K'}{2})}
    { \sn (\dfrac{\beta + \beta_k + i K'}{2})}=1
  \end{equation}
  for every internal vertex.
\end{theorem}
\begin{proof}
Condition (\ref{eq:spherical}) is equivalent to the fact that the angles $\theta_k$ around any internal vertex $v$ on a sphere sum up to~$2 N \pi$. This guarantees that the corresponding rings build a flower.
\end{proof}

Fort a more detailed analysis of the angles around vertexes let us introduce a smooth function of $x\in {\mathbb R}$
\begin{equation}
\label{eq:g(x)}
g(x):=\frac{\pi}{2}-\arg \sn(\frac{x+iK'}{2}), 
\end{equation}
normalized by $g(0)=0$. The function $g(x)$ possesses the following symmetries:
\begin{eqnarray}
g(x+4K)=g(x)+\pi,
\label{eq:g(x)_periodicity}
\\
g(-x)=-g(x).
\label{eq:g(x)_odd}
\end{eqnarray}
An example of the graph of this function is presented in Fig.~\ref{fig:g(x)}.
\begin{figure}[t]
  \centering
  \includegraphics[width=.6\linewidth]{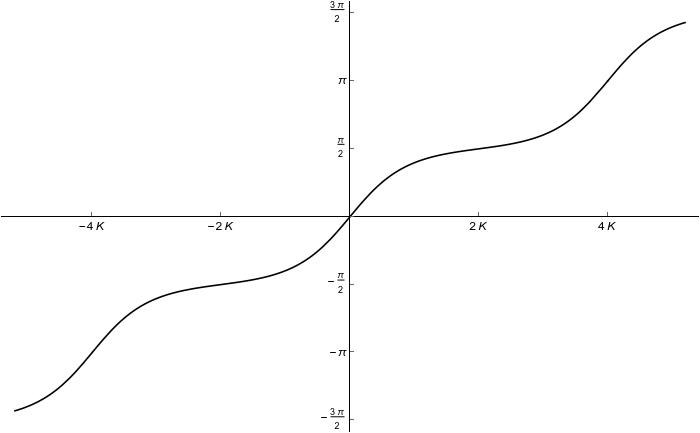}
  \caption{Graph of the function $g(x)=\frac{\pi}{2}-\arg \sn(\frac{x+iK'}{2})$ for $q=0.8$.
  }
  \label{fig:g(x)}
\end{figure}

Applying the addition formula 
$$
\sn(a+b)=\frac{\sn a\cn b\dn b+ \sn b\cn a \dn a}{1-q^2\sn^2 a\sn^2 b}
$$
and
\begin{equation}
\label{eq:iK'/2}
\sn\frac{iK'}{2}=\frac{i}{\sqrt{q}},\quad \cn\frac{iK'}{2}=\frac{\sqrt{1+q}}{\sqrt{q}},\quad \dn \frac{iK'}{2}=\sqrt{1+q},
\end{equation} 
we obtain another formula for $g(x)$:
\begin{equation}
\label{eq:g(x)_2}
g(x)=\arctan \frac{(1+q)\sn\frac{x}{2}}{\cn\frac{x}{2}\dn\frac{x}{2}}.
\end{equation}

The function $g(x)$ is monotonically increasing. Indeed, differentiating
$$
\sn\frac{x+iK'}{2}=\frac{i}{\sqrt{q}}e^{-ig(x)}
$$
implies
$$
g'(x)=-\dfrac{\cn\frac{x+iK'}{2}\dn\frac{x+iK'}{2}}{2i\sn\frac{x+iK'}{2}}=\frac{1}{2}\dn\frac{x+iK'}{2}\dn\frac{x-iK'}{2}.
$$
Applying the addition formula
$$
\dn\frac{x+iK'}{2}\dn\frac{x-iK'}{2}=\dfrac{\dn^2\frac{iK'}{2}-q^2\cn^2\frac{iK'}{2}\sn^2\frac{x}{2}}{1-q^2\sn^2\frac{iK'}{2}\sn^2\frac{x}{2}}
$$
and substituting (\ref{eq:iK'/2})
we obtain
\begin{equation}
\label{eq:g'_1}
g'(x)=\frac{1+q}{2}\left(\frac{1-q\sn^2\frac{x}{2}}{1+q\sn^2\frac{x}{2}}\right).
\end{equation}
This function is obviously positive valued for any $x$.

Another formula for it, which can be obtained by applying addition formulas, is
\begin{equation}
\label{eq:g'_2}
g'(x)=\frac{1}{2}(\dn x +q \cn x).
\end{equation}

For the second derivative we have
\begin{equation}
\label{eq:g''}
g''(x)=-\frac{\sn x}{2}(\dn x +q \cn x),
\end{equation}
which implies that $g(x)$ is concave on $x\in [0,2K]$ and convex on $x\in [2K,4K]$.


In the rest of this section we consider the most interesting case when every ring lies in a half-sphere, i.e.
$$
R<\frac{\pi}{2},\quad \text{for all rings},
$$
which corresponds to $\beta\in [0,2K]$. Four possible combinations of the orientations of two neighboring rings are presented in Fig.~\ref{fig:four_cases_orientatioin}. For $r>0$ the angle is positive $\theta_k>0$. For $r<0$ the angle is negative $\theta_k<0$. In both cases 
\begin{equation}
\label{eq:angle_tmp}
-\pi<\theta_k<\pi.
\end{equation}

\begin{figure}[t]
  \centering
  \includegraphics[width=\linewidth]{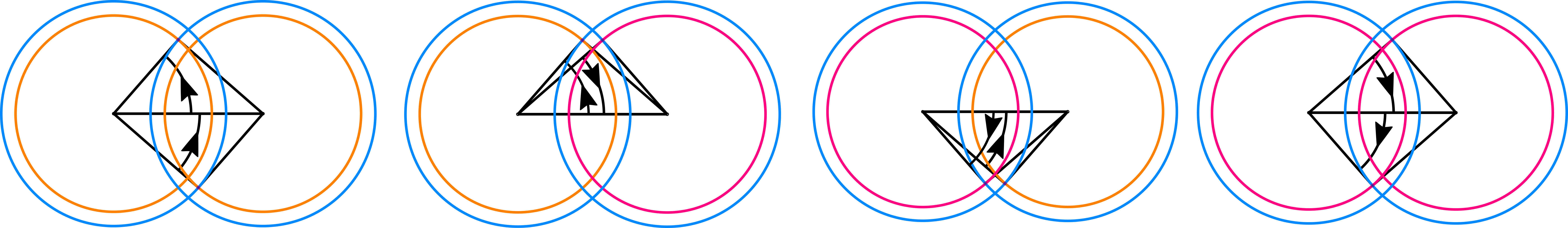}
  \caption{
    Cyclic quadrilaterals defined by two orthogonally intersecting circle
    rings depending on the signs of the radii: \\
    (Left): $r>0, r_k > 0$, embedded quadrilateral, $\theta_k>0$,\\
    (Center-Left): $r>0, r_k <0$, non-embedded quadrilateral, $\theta_k>0$,\\
    (Center-Right): $r<0, r_k >0$, non-embedded quadrilateral, $\theta_k<0$,
    (Right): $r<0, r_k <0$, embedded quadrilateral, $\theta_k<0$.
  }
  \label{fig:four_cases_orientatioin}
\end{figure}

\begin{lemma}
\label{lem:angle_spherical}
For two orthogonally intersecting rings on a sphere with external radii $R,R_k<\frac{\pi}{2}$ and for any orientation of their inner circles  $c$ and $c_k$ the opening angle $ \theta_k$ is given by
\begin{eqnarray}
\theta_k &=& \pi +g(\beta-\beta_k) - g(\beta+\beta_k), \quad \text{if}\  r>0,
\label{eq:angle_positive_r_exact}\\
\theta_k &=& g(\beta-\beta_k) - g(\beta+\beta_k), \quad \text{if}\  r<0.
\label{eq:angle_negative_r_exact}
\end{eqnarray}
\end{lemma}

\begin{proof}
Formulas (\ref{eq:angle_positive_r},\ref{eq:angle_negative_r})  coincide with (\ref{eq:angle_positive_r_exact},\ref{eq:angle_negative_r_exact}) modulo $2\pi$. On the other hand, since $\beta_k\in [0,2K]$, expressions (\ref{eq:angle_positive_r_exact},\ref{eq:angle_negative_r_exact}) satisfy the geometric condition (\ref{eq:angle_tmp}).
\end{proof} 
The cone angle at a vertex is defined by
\begin{equation}
\label{eq:Theta}
\Theta:=\sum_{k}\theta_k,
\end{equation}
where the sum is taken over all neighboring rings.
For an internal vertex of an orthogonal ring pattern $\Theta=2\pi$ for positive $(r>0)$ and  $\Theta=-2\pi$ for negative $(r<0)$ orientations respectively.

\begin{theorem}
\label{thm:beta-equations}
Uniformizing variables $\beta$ determine an orthogonal ring pattern with $R<\frac{\pi}{2}$ if and only if they lie in the interval $\beta\in [0,2K]$  and for all internal vertices satisfy the condition
\begin{equation}
\label{eq:angle_internal_beta}
\sum_{k=1}^4 g(\beta+\beta_k)-g(\beta-\beta_k)=2\pi,
\end{equation}
where the sum is taken over all four neighboring rings.

For a boundary vertice $v$ the variable $\beta$ satisfies 
\begin{eqnarray}
\label{eq:angle_boundary_beta}
\pi V(v)+\sum g(\beta-\beta_k)-g(\beta+\beta_k)=\Theta (v), \quad \text{if}\  r>0, \\
\sum g(\beta-\beta_k)-g(\beta+\beta_k)=\Theta (v), \quad \text{if}\  r<0, \nonumber
\end{eqnarray}
where $\Theta (v)$ is the total nominal angle at the vertex, and $V(v)$ is the valence of $v$, i.e. the number of neighboring rings orthogonal to the ring centered at $v$. 
\end{theorem}

\section{Hyperbolic orthogonal ring patterns}
\label{sec:hyperbolic_ring_patterns}

From the hyperbolic Pythagoras' Theorem we have:
$$
\cosh R \cosh r_k=\cosh r \cosh R_k,
$$
and it follows, as in the spherical
case,  that we can introduce a constant $q \le 1$ such that:
\begin{equation}
\label{eq:q_hyperbolic}
  \cosh r = q \cosh R.
\end{equation}

As in the spherical case we introduce new variables and express the radii
in terms of Jacobi elliptic functions:
\begin{equation}
\label{eq:hyperbolic_ring_uniformization_beta}
\cosh R=\sn (\beta,q) ,\ \sinh R= i \cn (\beta,q), \ \sinh r=i \dn (\beta, q).
\end{equation}
Note that in contrast to the spherical case, the $\beta$-variables are now
complex:
$$
\beta = \gamma + i K',  \gamma\in \R \quad {\rm for} \ \sn \beta \in \R, \sn \beta \ge 1/q.
$$ 
In terms of real parameters $\gamma$ the uniformization (\ref{eq:sphere_ring_uniformization}) is given by
\begin{equation}
\label{eq:hyperbolic_ring_uniformization_gamma}
\cosh R=\frac{1}{q\sn (\gamma,q)} ,\ \tanh R= \dn (\gamma,q), \ \sinh r=\frac{\cn (\gamma, q)}{\sn (\gamma,q)}.
\end{equation}

\begin{proposition}
\label{prop:hyperbolic_ring_uniformization}
The formulas (\ref{eq:hyperbolic_ring_uniformization_beta}) or (\ref{eq:hyperbolic_ring_uniformization_gamma})
describe a one to one correspondence between the radii $(r,R)\in [-\infty, +\infty] \times [R_0,+\infty]$ of a hyperbolic ring with the parameter $q\le 1$ given by (\ref{eq:q_hyperbolic}) and the variable $\beta=\gamma +iK'$ with $\gamma\in [0, 2K] $ of Jacobi functions of modulus $q$. Here $\cosh R_0= 1/q$. We control the sign by requiring $r=R$ for $q=1$.
\end{proposition}

\begin{proof}
The first two equations relate $R$ and $\gamma\in [0, 2K] $.  
Then the last formula in (\ref{eq:hyperbolic_ring_uniformization_gamma}) determines $r$, which satisfies $\cosh r= \frac{1}{\sn \gamma}$. This implies (\ref{eq:q_hyperbolic}). We note that $r$ may be negative. 
\end{proof}

The angles in a right-angled hyperbolic triangle may be computed using the
following formulas:
\begin{align*}
 \varphi_k = \arctan \frac{\tanh R_k}{\sinh r}
 \quad \text{and} \quad   
\psi_k  = \arctan \frac{\tanh r_k}{\sinh R}.
\end{align*}

In terms of the elliptic function parametrization the angles are given by
\begin{align}
  \varphi_k &= 
  \arctan\left( 
    \frac{\cn \beta_k}{\sn \beta_k \dn \beta}
  \right)
	    &\text{and}&&
  \psi_k &=
  \arctan\left(
  \frac{\dn \beta_k}{q\sn \beta_k \cn \beta}
  \right).
\end{align}

They can be transformed to the following form
\begin{eqnarray*}
  \varphi_k &=& \arg \left(1-\frac{\dn (\beta_k+iK')}{\dn \beta}\right),\\
  \psi_k &= &\arg \left(1-\frac{\cn (\beta_k+iK')}{\cn \beta}\right). 
\end{eqnarray*}

For the sum this implies
$$
\theta_k=\varphi_k+\psi_k= \arg \left( ( \frac{-1}{\cn\beta \dn\beta}) (\frac{\cn\beta-\cn (\beta_k+iK')}{\dn\beta +\dn (\beta_k+iK')})\right).
$$
Note that $ \cn (\beta_k+iK')$ is real and $\cn (\beta)$ is pure imaginary. 

The first factor in this formula is positive for $r>0$ and negative for $r<0$. Applying the addition formulas (\ref{eq:addition_formulas}) we obtain
\begin{equation}
\label{eq:theta_beta_hyp}
\theta_k=\arg\left(-\text{sign}(r)\frac{\sn\frac{\beta+\beta_k+ iK'}{2}}{\sn\frac{\beta-\beta_k+ iK'}{2}}\right).
\end{equation}
It turns on that orthogonal hyperbolic and spherical ring patterns are governed by the same complex equation. The following theorem is analog to theorem~\ref{thm:spherical}. 
\begin{theorem} 
\label{thm:hyperbolic}
Rings build a generalized orthogonal ring pattern in hyperbolic plane if and only if they are given by the variables $\beta=\gamma+iK'$ with $\gamma\in[0,2K]$ satisfying equation (\ref{eq:spherical}) for every internal vertex of $G$.
  \end{theorem}
  
  In $\gamma$-variables the governing equation for hyperbolic patterns is
   \begin{equation}  \label{eq:hyperbolic}
    \prod_{v_k \sim v} 
  \dfrac{  \sn (\dfrac{\gamma - \gamma_k + i K'}{2})}{ \sn (\dfrac{\gamma + \gamma_k - i K'}{2})}=1.
  \end{equation}

Formula (\ref{eq:theta_beta_hyp}) implies
\begin{eqnarray*}
  \theta_k &=& -\arg\left(\sn\frac{\gamma+\gamma_k+ iK'}{2} \sn\frac{\gamma-\gamma_k+ iK'}{2}\right) +\pi, 
  \quad \text{if}\ r>0, 
  \label{eq:angle_positive_r_hyp}\\
 \theta_k &=& -\arg\left(\sn\frac{\gamma+\gamma_k+ iK'}{2} \sn\frac{\gamma-\gamma_k+ iK'}{2}\right), 
  \quad \text{if}\ r<0. 
   \label{eq:angle_negative_r_hyp}
 \end{eqnarray*}
 
 \begin{lemma}
\label{lem:angle_hyperbolic}
For two orthogonally intersecting hyperbolic rings and for any orientation of their inner circles  $c$ and $c_k$ the opening angle $ \theta_k$ is given by
\begin{eqnarray}
\theta_k &=& g(\gamma-\gamma_k) + g(\gamma+\gamma_k), \quad \text{if}\  r\ge 0 \Leftrightarrow \gamma\in [0, K],
\label{eq:angle_positive_r_exact_gamma}\\
\theta_k &=& g(\gamma-\gamma_k) + g(\gamma+\gamma_k)- \pi, \quad \text{if}\  r\le 0  \Leftrightarrow \gamma\in [K, 2K],
\label{eq:angle_negative_r_exact_gamma}
\end{eqnarray}
where $g(x)$ is the function (\ref{eq:g(x)}).
\end{lemma}

Analogously to the spherical case we obtain
\begin{theorem}
\label{thm:gamma-equations}
Uniformizing variables $\gamma$ determine a hyperbolic orthogonal ring pattern if and only if they lie in the interval $\gamma\in [0,2K]$  and for all internal vertices satisfy the condition
\begin{equation}
\label{eq:angle_internal_gamma}
\sum_{k=1}^4 g(\gamma+\gamma_k)+g(\gamma-\gamma_k)=2\pi,
\end{equation}
where the sum is taken over all four neighboring rings. Here $g(x)$ is the function (\ref{eq:g(x)}).

For a boundary vertice $v$ the variable $\gamma$ satisfies 
\begin{eqnarray}
\label{eq:angle_boundary_gamma}
\sum g(\gamma-\gamma_k)+g(\gamma+\gamma_k)=\Theta (v), \quad \text{if}\  r>0,  \\
\sum g(\gamma-\gamma_k)+g(\gamma+\gamma_k)=\Theta (v)+ \pi V(v), \quad \text{if}\  r<0, \nonumber
\end{eqnarray}
where $\Theta (v)$ is the total nominal angle at the vertex (positive for $r>0$ and negative for $r<0$), and $V(v)$ is the valence of $v$, i.e. the number of neighboring rings orthogonal to the ring centered at $v$. 
\end{theorem}

\section{Circle pattern limit}
\label{sec:cp_limit}

For $q=1$ both circles of the ring coincide $R=r$, and we obtain orthogonal circle patterns in the sphere and in the hyperbolic plane.
The Jacobi elliptic functions become hyperbolic functions in this case:
\begin{equation}
\label{eq:Jacobi_q=1}
\sn(z,1)=\tanh z,\ \cn (z,1)=\dn (z,1)=\frac{1}{\cosh z},\ K'=\frac{\pi}{2}, \ K\to \infty.
\end{equation}

The function $g(x)$ given by (\ref{eq:g(x)_2}) also simplifies:
$$
g(x)=\arctan \sinh x=-\frac{\pi}{2}+2\arctan e^x.
$$
The last identity follows from the addition formula
$$
\arctan a +\arctan b =\arctan \frac{a+b}{1-ab}.
$$


The relation between the $\beta$-variables and the spherical
radii is the following:
\begin{equation}
\label{eq:q=1_spherical}
  \tanh \beta = \cos R, \frac{1}{\cosh \beta}=\sin R \quad\Rightarrow\quad e^\beta = \cot \frac{R}{2}.
\end{equation}

The circles are positively oriented, and formula (\ref{eq:angle_positive_r_exact}) for the angles $\theta_k$ becomes
\begin{eqnarray*}
\theta_k=\pi + 2\arctan e^{\beta-\beta_k}- 2\arctan e^{\beta +\beta_k}=
2\arctan e^{\beta-\beta_k}+2\arctan e^{-\beta -\beta_k}.
\end{eqnarray*}


In the case $q=1$, and $\beta=\gamma +\frac{\pi}{2}$ formulas (\ref{eq:hyperbolic_ring_uniformization_gamma}) give the following representation for the hyperbolic radii:
\begin{equation}
\label{eq:q=1_hyperbolic}
\cosh R= \coth \gamma, \ \sinh R= \frac{1}{\sinh \gamma} \quad\Rightarrow\quad e^\gamma = \coth \frac{R}{2}.
\end{equation}

The circles are positively oriented, and formula (\ref{eq:angle_positive_r_exact_gamma}) becomes
\begin{eqnarray*}
\theta_k=
2\arctan e^{\gamma_k-\gamma}+2\arctan e^{-\gamma -\gamma_k}.
\end{eqnarray*}
In this limiting case the variables $\beta$ and $\gamma$ differ by the sign from the canonical variables $\rho$ used in \cite{BHoSp06} and \cite{BSp04}.

\section{Variational description}
\label{sec:variational}
Orthogonal ring patterns can be described using a variational principle.

Let us consider the anti-derivative of $g(x)$
\begin{eqnarray}
\label{eq:F(x)}
& & F(x)= \int_0^x g(u)du=\\
&  &\int_0^x \frac{\pi}{2} -\arg\sn\frac{u+iK'}{2} du=\int_0^x \arctan \frac{(1+q)\sn\frac{u}{2}}{\cn\frac{u}{2}\dn\frac{u}{2}}du.
\nonumber
\end{eqnarray}
$F(x)$ is an even convex function with $F(0)=0, F'(x)=g(x)$ and
\begin{equation}
\label{eq:F''}
F''(x)=\frac{1}{2}(\dn x+q\cn x),
\end{equation}
see (\ref{eq:g'_2}). Its graph for $q=0.8$ is shown in Fig.~\ref{fig:F(x)}.
\begin{figure}[t]
  \centering
  \includegraphics[width=.4\linewidth]{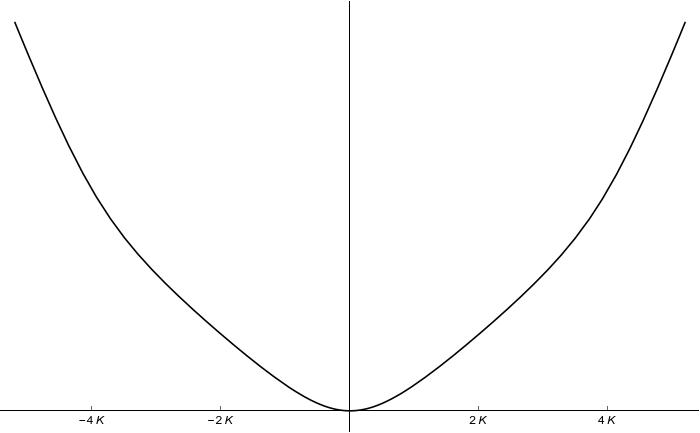}
  \caption{Graph of $F(x)$ for $q=0.8$.}
  \label{fig:F(x)}
\end{figure}

\subsection*{In the sphere}

Consider the functional
\begin{equation}
\label{eq:functional_spherical}
S_{sph}(\beta):=\sum_{(jk)} \left( F(\beta_j-\beta_k)-F(\beta_j+\beta_k)\right) +\sum_j\Phi_j\beta_j,
\end{equation}
where the first sum is taken over all edges and the second sum over all vertices of $G$, and $\Phi_j$ are some prescribed parameters at the vertices. 

\begin{theorem}
\label{thm:variational_spherical}

The critical points $\beta$ of the functional $S_{sph}$ with all $\beta_j\in [0,2K]$ and 
\begin{eqnarray}
&\Phi_j=2\pi & \text{for  inner rings}, \nonumber \\
&\Phi_j=\pi V(j)-\Theta(j) &  \text{for positively oriented boundary rings},
\label{eq:Phi_j_sph}     \\
&\Phi_j=-\Theta(j) &  \text{for negatively oriented boundary rings}. \nonumber
 \end{eqnarray}
 correspond to spherical orthogonal ring patterns with all $R_j\le\frac{\pi}{2}$.
 Here $V(j)$ and $\Theta(j)$ are the number of  the rings neighboring to the $j$-th boundary ring and  the nominal angle covered by these rings respectively. 
 
The second derivative of $S_{sph}$ is the quadratic form
 \begin{eqnarray}
 \label{eq:hessian_spherical}
 D^2 S_{sph}= &\frac{1}{2}\sum_{(j,k)} (\dn(\beta_j-\beta_k)+q\cn(\beta_j-\beta_k))(d\beta_j-d\beta_k)^2-\\
&(\dn(\beta_j+\beta_k)+q\cn(\beta_j+\beta_k))(d\beta_j+d\beta_k)^2,\nonumber
 \end{eqnarray}
 where the sum is taken over pairs of neighboring rings.
 \end{theorem}

\begin{proof}
We obtain 
$$
\frac{\partial S_{sph}(\beta)}{\partial\beta_j}=\sum_{\text{neighbors}\ k} \left(g(\beta_j-\beta_k)-g(\beta_j+\beta_k)\right) +\Phi_j,  
$$
which coincides with (\ref{eq:angle_internal_beta}, \ref{eq:angle_boundary_beta}) for $\Phi_j$ chosen by (\ref{eq:Phi_j_sph}).

Equation (\ref{eq:hessian_spherical}) for the second derivative of $S_{sph}$ is obtained by taking the second derivative in the first sum of equation (\ref{eq:functional_spherical}).
\end{proof}

Orthogonal ring patterns with the prescribed boundary angles  $\Theta(j)$ and all rings satisfying $R<\frac{\pi}{2}$ are in one to one correspondence with the critical points of this functional lying in the interval $\beta\in [0, 2K]$. 

In the case of orthogonal circle patterns $q=1$ we obtain
$$
F(x)=\int_0^x \arctan \sinh u du= -\frac{\pi}{2}x+2\int_0^x\arctan e^u du.
$$
Using its evenness it can be transformed to the form 
$$
F(x)={\mathcal F}(x)+{\mathcal F}(-x)-2{\mathcal F}(0),
$$
where
\begin{equation}
\label{eq:functional_dilogarithm}
{\mathcal F}(x)=\int_{-\infty}^x \arctan e^u du= {\rm Im} {\rm Li}_2(ie^x). 
\end{equation}
The dilogarithm function is defined by ${\rm Li}_2(z)=-\int_0^z \log(1-\xi)d\xi/\xi$.

This gives the functional 
\begin{equation}
\label{eq:functional_spherical_circles}
S_{sph}=\sum_{(jk)} \left( {\mathcal F}(\beta_j-\beta_k)+{\mathcal F}(\beta_k-\beta_j)-
{\mathcal F}(\beta_j+\beta_k)- {\mathcal F}(-\beta_j-\beta_k)\right) +\sum_j\Phi_j\gamma_j, 
\end{equation}
used in \cite{BHoSp06} for orthogonal circle patterns (note that $\beta=-\rho$). Its second derivative is
\begin{eqnarray*}
 D^2 S_{sph}= \sum_{(j,k)} \frac{1}{\cosh(\beta_j-\beta_k)}(d\beta_j-d\beta_k)^2-
 \frac{1}{\cosh(\beta_j+\beta_k)}(d\beta_j+d\beta_k)^2.
 \end{eqnarray*} 

\subsection*{In the hyperbolic plane}

Variational description of hyperbolic ring patterns is similar to the spherical case. Let us introduce the functional
\begin{equation}
\label{eq:functional_hyperbolic}
S_{hyp}(\gamma):=\sum_{(jk)} \left( F(\gamma_j-\gamma_k)+F(\gamma_j+\gamma_k)\right) +\sum_j\Phi_j\gamma_j.
\end{equation}

Denote by $V_B$ the set of boundary vertices of $G$, i.e. the vertices with less then four neighbors.  
\begin{theorem}
\label{thm:variational_hyperbolic}
Let $\Theta:V_B\to (-2\pi, 2\pi)$ be the prescribed cone angles at the boundary vertices. Assume that 
$|\Theta(j)|<\pi$ for the boundary vertices of valence 1. Then
the critical points of the functional $S_{hyp}$ with
\begin{eqnarray}
&\Phi_j=-2\pi & \text{for  inner rings}, \nonumber \\
&\Phi_j=  -\Theta(j) & \text{for positively oriented boundary rings},
\label{eq:Phi_j_hyp}     \\
&\Phi_j=-\pi V(j)-\Theta(j) &  \text{for negatively oriented boundary rings}, \nonumber
 \end{eqnarray}
 correspond to hyperbolic orthogonal ring patterns. 
Here $V(j)$ is the number of  the rings neighboring to the $j$-th boundary ring.

In particular, if $\gamma$ is a critical point of $S_{hyp}$ then all $\gamma_j\in[0,2K]$. 
 \end{theorem}
 \begin{proof}
The equation for critical points
$$
\frac{\partial S_{hyp}(\gamma)}{\partial\gamma_j}=\sum_{\text{neighbors}\ k} \left(g(\gamma_j-\gamma_k)+g(\gamma_j+\gamma_k)\right) +\Phi_j =0 
$$
coincides with (\ref{eq:angle_internal_gamma}, \ref{eq:angle_boundary_gamma}).

It is left to show that if $\gamma$ is a critical point of $S_{hyp}$ then all $\gamma_j$ lie in the interval $[0,2K]$. 
The function $g(x)$ is monotonically increasing. Therefore for $\gamma_j>2K$ using the symmetries (\ref{eq:g(x)_periodicity}, \ref{eq:g(x)_odd}) we obtain
$$
g(\gamma_j -\gamma_k)+g(\gamma_j + \gamma_k)>g(2K-\gamma_k)+g(2K+\gamma_k)=-g(-2K+\gamma_k)+g(2K+\gamma_k)=\pi.
$$
Thus we have
$$
\frac{\partial S_{hyp}(\gamma)}{\partial\gamma_j} > V(j)\pi +\Phi_j \ge 0
$$
for $\Phi_j\ge -V(j)\pi$. The last condition is automatically satisfied for inner vertices as well as for boundary vertices with negatively oriented rings. For positively oriented boundary vertices the condition $\Theta(j) \le V(j)\pi$ is satisfied for all vertices with $V(j) \ge 2$. For the boundary vertices of valence 1 we obtain the condition $\Theta(j)\le \pi$.

For $\gamma_j<0$ we obtain
$$
g(\gamma_j -\gamma_k)+g(\gamma_j + \gamma_k)< g(-\gamma_k)+g(\gamma_k)=0.
$$
Thus  we have
$$
\frac{\partial S_{hyp}(\gamma)}{\partial\gamma_j} <\Phi_j \le 0.
$$
The last inequality  $\Phi_j\le 0$ is satisfied for inner vertices as well as for boundary vertices with positively oriented rings. For negatively oriented boundary vertices the condition $-V(j)\pi-\Theta(j) \le 0$ is satisfied for all vertices with $V(j) \ge 2$. For the boundary vertices of valence 1 we obtain the condition $|\Theta(j)|\le \pi$.
\end{proof}

Next theorem follows from the fact that the derivative $g'(x)$ is always positive, see (\ref{eq:g'_1}) and (\ref{eq:g'_2}). 
\begin{theorem}
\label{thm:convex_hyperbolic}
The functional $S_{hyp}$ is convex, its second derivative  is the quadratic form
 \begin{eqnarray}
 \label{eq:hessian_hyperbolic}
 D^2 S_{hyp}= &\frac{1}{2}\sum_{(j,k)} (\dn(\gamma_j-\gamma_k)+q\cn(\gamma_j-\gamma_k))(d\gamma_j-d\gamma_k)^2+\\
&(\dn(\gamma_j+\gamma_k)+q\cn(\gamma_j+\gamma_k))(d\gamma_j+d\gamma_k)^2,\nonumber
 \end{eqnarray}
 where the sum is taken over pairs of neighboring rings.
 \end{theorem}

The convexity of $S_{hyp}$ allows to construct hyperbolic ring patterns by simple minimization of the functional and to prove their existence and uniqueness, see Section~\ref{sec:computation}.

In the case $q=1$  the functional $S_{hyp}$ up to a constant becomes the functional for hyperbolic orthogonal circle patterns, introduced in \cite{BSp04}: 
$$
S_{hyp}=\sum_{(jk)} \left( {\mathcal F}(\gamma_j-\gamma_k)+{\mathcal F}(\gamma_k-\gamma_j)+
{\mathcal F}(\gamma_j+\gamma_k)+ {\mathcal F}(-\gamma_j-\gamma_k)\right) +\sum_j\Phi_j\gamma_j, 
$$
where $\mathcal F$ is given by (\ref{eq:functional_dilogarithm}).

The second derivative of this convex functional is
\begin{eqnarray*}
 D^2 S_{hyp}= \sum_{(j,k)} \frac{1}{\cosh(\gamma_j-\gamma_k)}(d\gamma_j-d\gamma_k)^2+
 \frac{1}{\cosh(\gamma_j+\gamma_k)}(d\gamma_j+d\gamma_k)^2.
 \end{eqnarray*} 

\section{Boundary value problems} 
\label{sec:computation}
Let us construct hyperbolic orthogonal ring patterns satisfying classical boundary conditions. 

\begin{theorem} 
\label{thm:BVP_hyperbolic}
Hyperbolic orthogonal ring patterns can be obtained as solutions of the following boundary valued problems:
\begin{itemize}
\item (Dirichlet boundary conditions)
For any choice of prescribed radii $\gamma:V_B\to [0,2K]$ of boundary rings there exists a unique hyperbolic orthogonal 
ring pattern $\mathcal R$ with these boundary radii.
\item (Neumann boundary conditions) 
Let $\Theta:V_B\to (-2\pi, 2\pi)$ be prescribed boundary cone angles such that $|\Theta(j)|<\pi$ for the boundary vertices of valence 1.
Then  there exists a unique hyperbolic orthogonal ring pattern $\mathcal R$ with these boundary cone angles.  
\end{itemize}
\end{theorem}
\begin{proof}
For the Dirichlet boundary value problem one minimizes the functional $S_{hyp}(\gamma)$ given by (\ref{eq:functional_hyperbolic}) with $\Phi_i=-2\pi$ for all internal vertices by varying $\gamma$ on internal vertices.
For the Neumann boundary value problem one minimizes the functional $S_{hyp}(\gamma)$ with the cone angles given by (\ref{eq:Phi_j_hyp}) by varying $\gamma$ on all vertices.
The uniqueness of the solution follows from the convexity of the functional  $S_{hyp}(\gamma)$, see Theorems~\ref{thm:variational_hyperbolic} and \ref{thm:convex_hyperbolic}.

To prove the existence one should show additionally that  
\begin{equation}
\label{eq:infinite_S}
S_{hyp}(\gamma)\to +\infty, \quad \text{for}\  |\gamma |_{\max}\to \infty,
\end{equation}
where $|\gamma |_{\max}=\max_j \{ |\gamma_j  |\}$. Then the minimum of $S_{hyp}$ is reached at finite $\gamma_j$'s.

Since $F(x)$ is convex an even, we have
$$
F(\gamma_j - \gamma_k)+F(\gamma_j + \gamma_k)> 2F(\gamma_j) \ \text{and}\ >2F(\gamma_k),
$$
thus
$$
F(\gamma_j - \gamma_k)+F(\gamma_j + \gamma_k)> F(\gamma_j)+F(\gamma_k).
$$
For the functional this implies
\begin{equation}
\label{eq:S>VF+linear}
S_{hyp}(\gamma)>\sum_j (V(j)F(\gamma_j)+\Phi_j \gamma_j).
\end{equation}
The function $g(x)$  can be well approximated by the linear function 
$$
\tilde{g}(x)=\frac{\pi x}{4K}, \ g(2nK)=\tilde{g}(2nK), \forall n\in {\mathbb Z}.
$$ 
Moreover since $g(x)$ is concave on $x\in [0, 2K]$ and 
$$
g(2K+x)-\frac{\pi}{2}=\frac{\pi}{2}-g(2K-x),
$$ 
we have for $x>0$,
$$
F(x)=\int_0^x g(u)du>\int_0^x \tilde{g}(u)du= \frac{\pi x^2}{8K}.
$$
Thus the first term in (\ref{eq:S>VF+linear}) grows at least quadratically in $|\gamma_j |$, which implies (\ref{eq:infinite_S}). 
\end{proof} 
 
The variational description on Theorem~\ref{thm:BVP_hyperbolic} gives an effective tool for computation of the corresponding ring patterns. The solution can be obtained by direct minimization of the functional. An example of solution of a Neumann boundary value problem is presented in Fig.~\ref{fig:variational_hyperbolic}. All images in this Section present ring patterns computed by Nina Smeenk using the described variational methods.

\begin{figure}[htbp]
  \centering
\includegraphics[width=.7\linewidth]{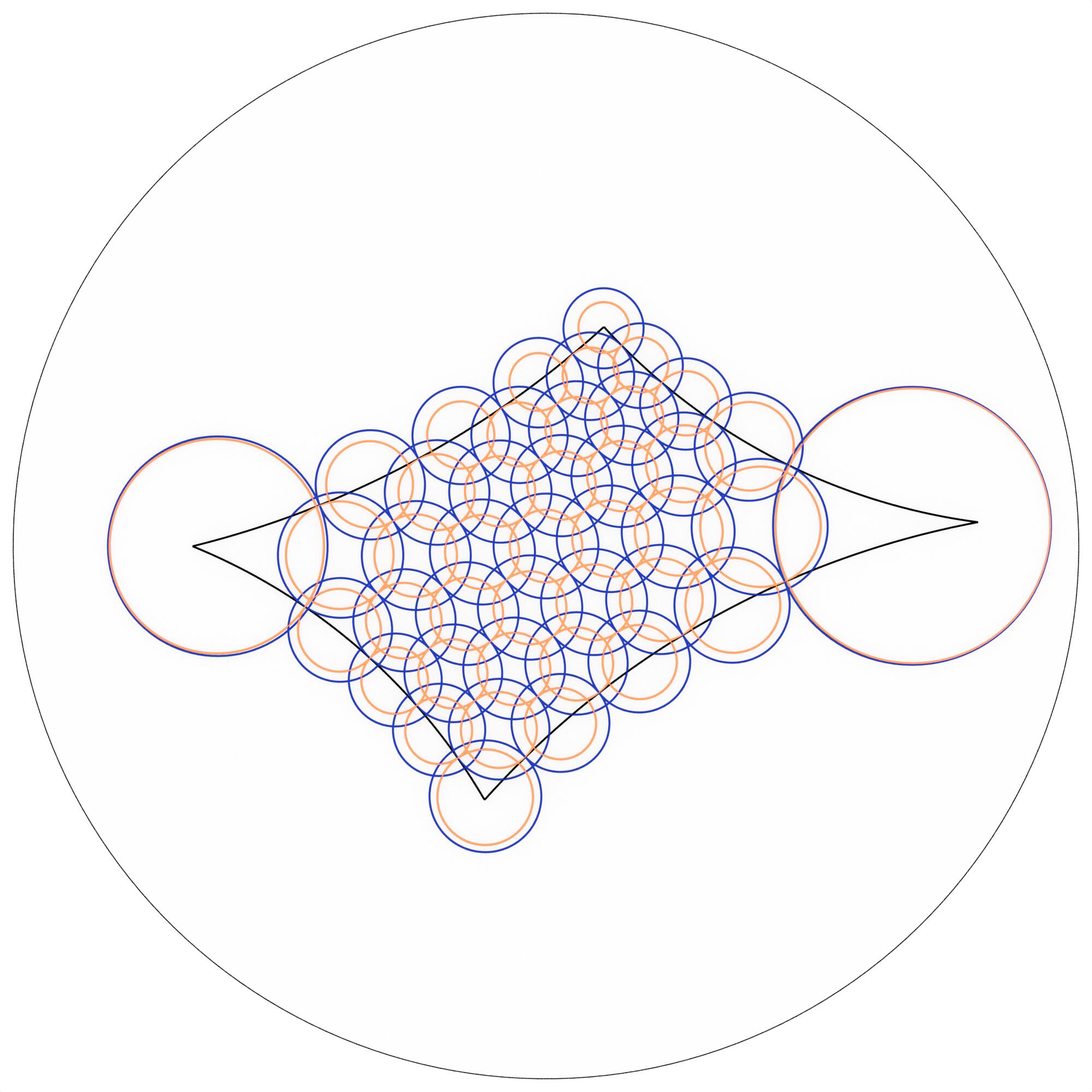}
  \caption{A hyperbolic orthogonal ring pattern solving the Neumann boundary value problem with a quadrilateral boundary,
   and $q=0.99$. The prescribed angles are
    $\pi$ for the boundary vertices of degree $2$. The shape is governed by the 
    four angles: $\Theta_1=\frac{\pi}{2}, \Theta_2=\frac{\pi}{10}, \Theta_3=\frac{2\pi}{5}, \Theta_4=\frac{\pi}{5}$, 
   with the sum $<2\pi$, prescribed for the
  four corner boundary vertices of degree 1.}%
   \label{fig:variational_hyperbolic} 
\end{figure}

Spherical orthogonal ring patterns also can be constructed using a variational principle. The method is similar to the one used in \cite{BHoSp06} for construction of orthogonal circle patterns in a sphere.

The functional (\ref{eq:functional_spherical}) is not convex. The quadratic form (\ref{eq:hessian_spherical})
is negative for the tangent vector $v=\sum_j \partial/\partial\beta_j$, the
index is therefore at least $1$. Define a reduced functional
$\widetilde{S}_{sph}(\beta)$ by maximizing in the direction $v$:
\begin{equation*}
  \widetilde{S}_{sph}(\beta)=\max_t S_{sph}(\beta+t v).
\end{equation*}
Obviously, $\widetilde{S}_{sph}(\beta)$ is invariant under translations in the
direction $v$. Now the idea is to minimize $\widetilde{S}_{sph}(\beta)$ restricted
to $\sum_j\beta_j=\rm{const}$. This method has proved to be amazingly
powerful. In particular, it can be used to produce branched orthogonal ring patterns
in the sphere.

\begin{figure}[htbp]
  \centering
\includegraphics[width=.3\linewidth]{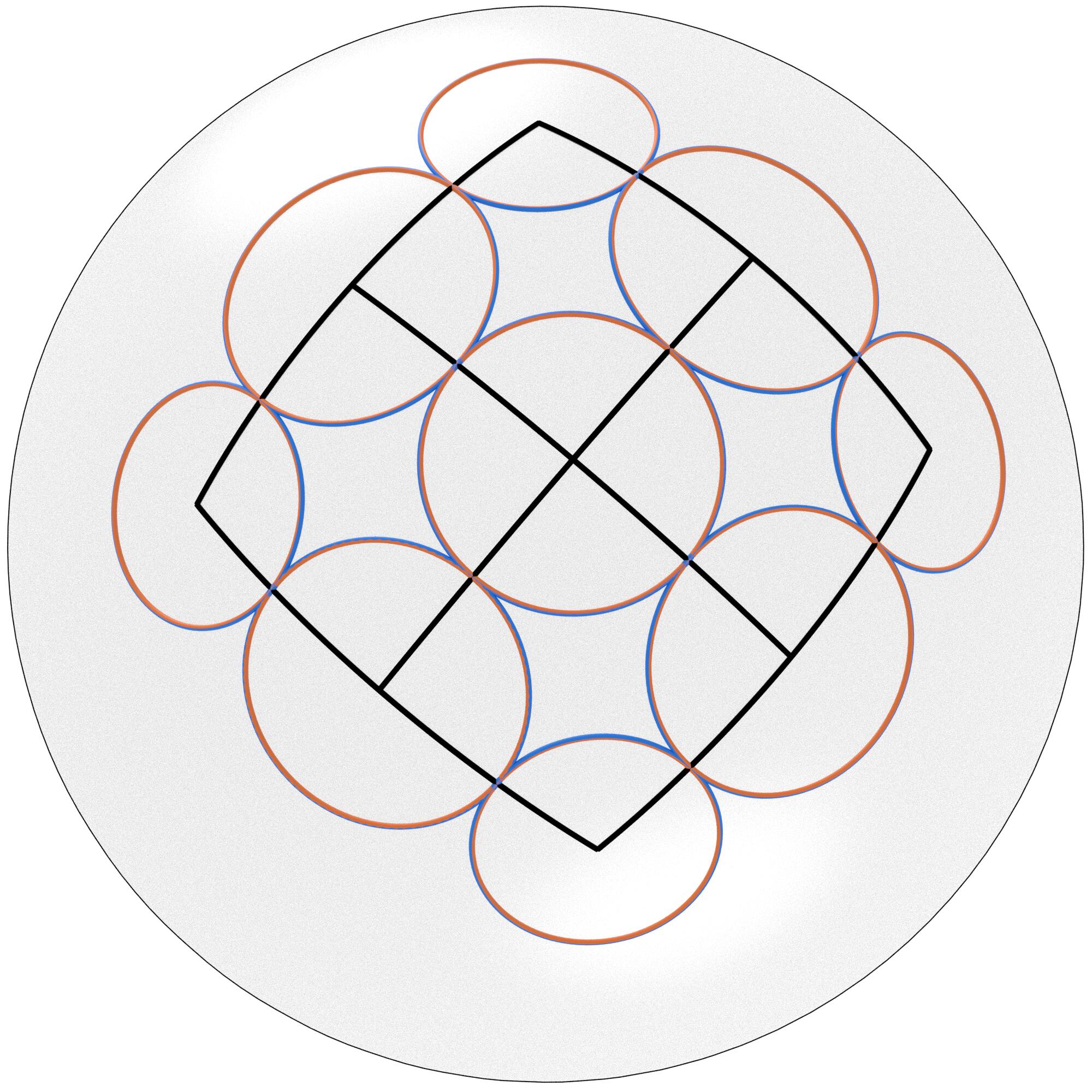}
\includegraphics[width=.3\linewidth]{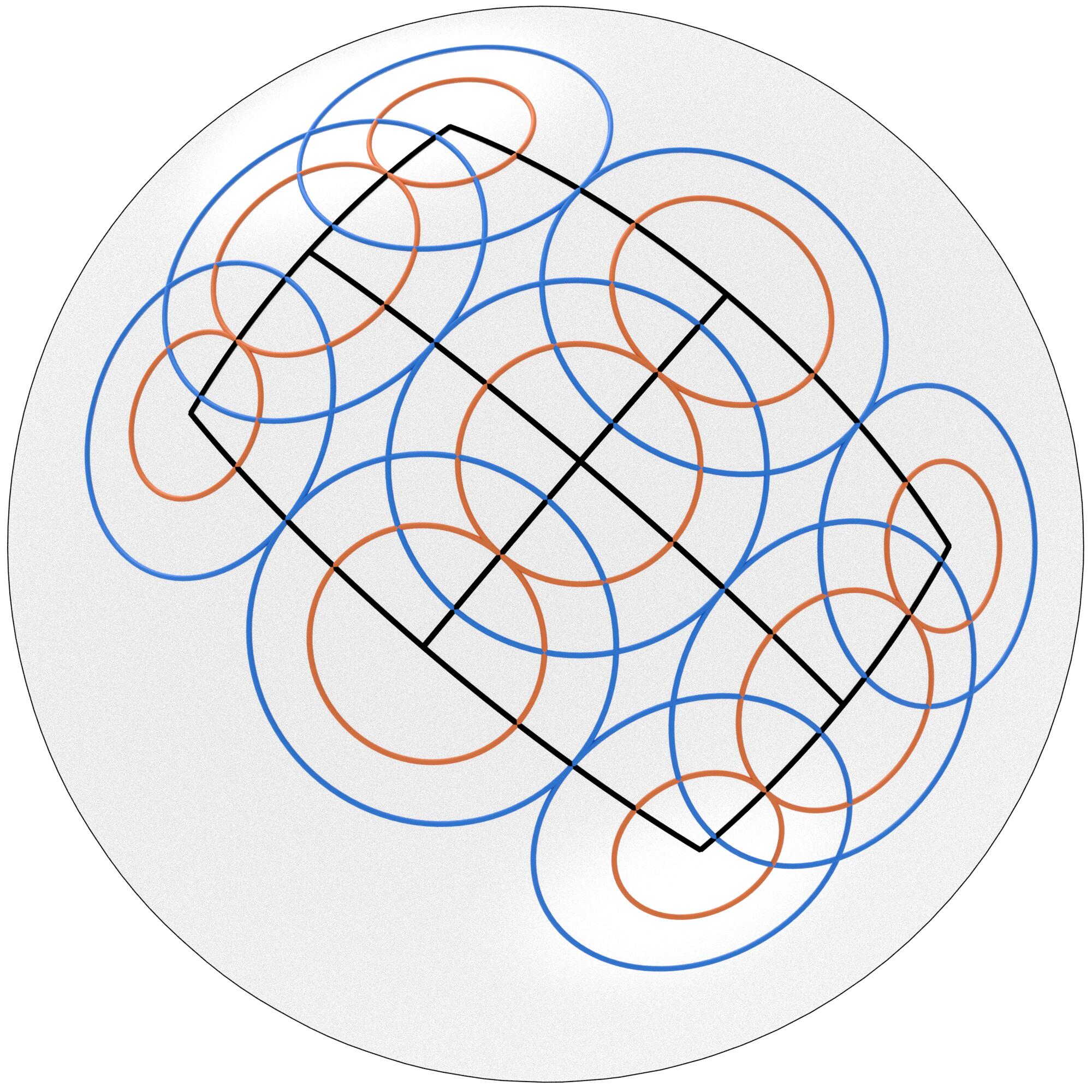}
\includegraphics[width=.3\linewidth]{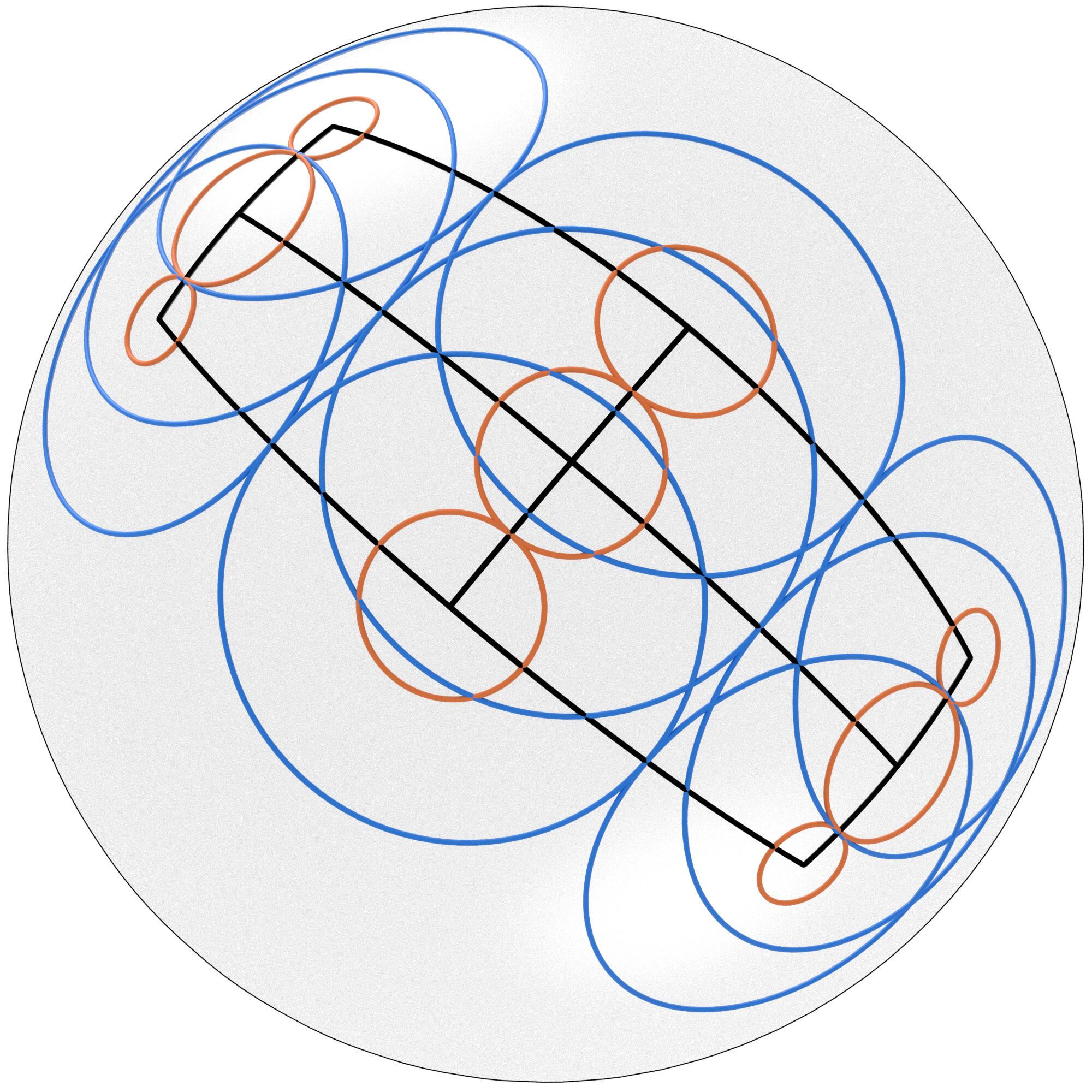}
  \caption{Spherical orthogonal ring patterns solving the Neumann boundary value problem with a quadrilateral boundary, and with $q=0.9999$ (left), $q=0.98$ (center) and $q=9.95$ (right) computed using the variational
    principle. The prescribed angles are
    $\pi$ for the boundary vertices of degree $2$. The shape is governed by the 
    four angles: $\Theta_1=\Theta_2=\Theta_3=\Theta_4=.55\pi$, with the sum $>2\pi$, prescribed for the
  four corner boundary vertices of degree 1.  For more clarity only half of the rings is shown.}%
   \label{fig:variational_spherical} 
\end{figure}

Three examples of solution of a Neumann boundary value problem with different parameters $q$ are presented in Fig.~\ref{fig:variational_spherical}). 
We minimize the functional  $\widetilde{S}_{sph}(\beta)$ for $q<1$ preserving the sum $\sum_j\beta_j=\rm{const}$. Numerical experiments show that the ring patterns exist for $q$ sufficiently close to 1. This fact is in complete agreement with the theory we present in Section~\ref{sec:deformations}.

Another example of an orthogonal ring pattern is presented in Fig.~\ref{fig:gauss_cmc}.  It corresponds to the Gauss map of a periodic discrete cmc surface, see \cite{BHoSm24}, which is a discrete analogue of the doubly periodic cmc surface $\psi_{0,\frac{\pi}{2}}(U_{2,2})$ by Lawson, see \cite[Theorem 9]{La70}. 

\begin{figure}[htbp]
  \centering
\includegraphics[width=.45\linewidth]{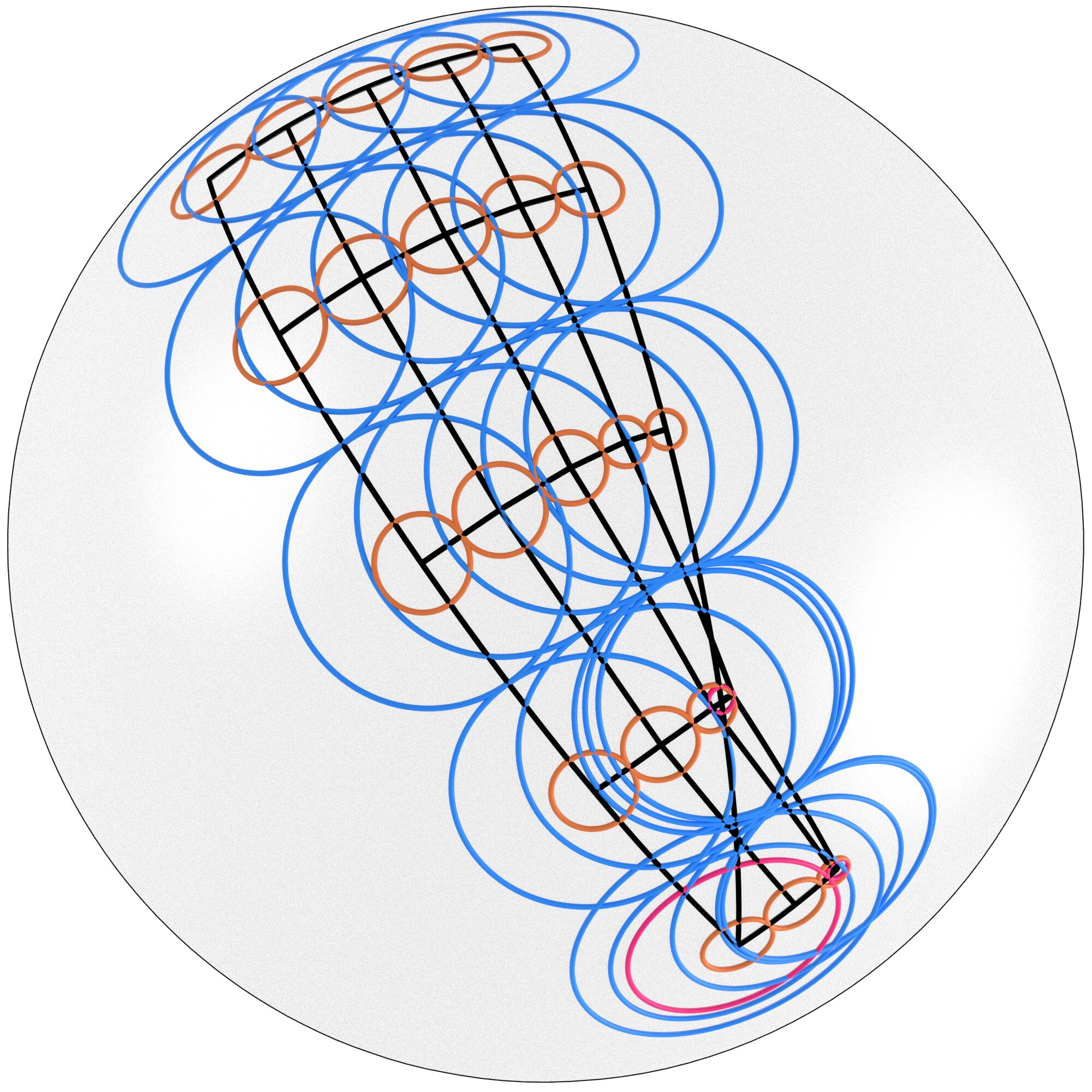}
\caption{An orthogonal ring pattern corresponding to the Gauss map of a discrete periodic cmc net in \cite{BHoSm24}, $q = 0.98288979$. The angles at the corner vertices are: $\Theta_1=\frac{2\pi}{3}, \Theta_2=\Theta_3=\Theta_4=\frac{\pi}{2}$. One half of the rings are shown. Red inner circles are negatively oriented.}%
   \label{fig:gauss_cmc} 
\end{figure}

\section{Ring patterns as deformations of circle patterns}
\label{sec:deformations}

Existence and uniqueness problems for spherical ring patterns are more difficult to investigate then in the hyperbolic case, since the corresponding functional $S_{sph}$ is not convex. However the variational description can be used to prove the existence of ring patterns that are close to circle patterns. Note that circle patterns in the sphere can be constructed by the stereographic projection of circles patterns in the plane. For the latter there exists a developed theory, including existence and uniqueness results, based on a semi-convex variational principle \cite{BSp04}.

Let us consider the Dirichlet or Neumann boundary value problems for orthogonal circle and ring patterns, analogous to the ones formulated in Theorem \ref{thm:BVP_hyperbolic} in the hyperbolic case. For both boundary value problems one has $N$ governing equations
$$
f_i(\beta)=0, \ \text{where} \ f_i(\beta):=\frac{\partial S_{sph}}{\partial \beta_i} 
$$
for $N$ variables $\beta_i$, where $N$ is the number of rings in the Neumann case and the number of inner rings in the Dirichlet case.

Consider an orthogonal spherical circle pattern which solves such a boundary value problem. It is parametrized by spherical radii $R^0_i$ of the circles or, equivalently, by the variables $\beta^0_i=\log \cot \frac{R^0_i}{2}$, see \eqref{eq:q=1_spherical} that satisfy the governing equations $f_i(\beta^0)=0$. Note that the functional \eqref{eq:functional_spherical_circles} for circle patterns is a special case of the functional $S_{sph}$ for ring patterns corresponding to $q=1$.

One says that a circle pattern allows an infinitesimal deformation if there exist $\phi_i$'s such that for the one-parameter family deformation 
$$
\beta_i^\epsilon:=\beta_i^0+\epsilon \phi_i
$$
the governing equations are satisfied up to the terms of second order $f_i(\beta^0+\epsilon \phi)=o(\epsilon)$, or, equivalently,
$$
0=\frac{\partial f_i(\beta^0 +\epsilon \phi)}{\partial \epsilon}_{| \epsilon=0} =\sum_j \frac{\partial f_i}{\partial\beta_j}(\beta^0) \phi_j=\sum_j \frac{\partial^2 S_{sph}}{\partial\beta_i\partial\beta_j}(\beta^0) \phi_j.
$$
This condition can be obviously reformulated as 
$$
\det \frac{\partial^2 S_{sph}}{\partial\beta_i \partial\beta_j}(\beta^0)=0.
$$
Circle patterns that do not allow infinitesimal deformations are called {\em rigid}. We have
\begin{proposition}
A circle pattern solution of a Dirichlet or Neumann boundary value problem with variables $\beta^0$ is rigid if and only if
\begin{equation}
\label{eq:Hess_nondeg} 
\det \frac{\partial^2 S_{sph}}{\partial\beta_i \partial\beta_j}(\beta^0)\neq 0.
\end{equation}
\end{proposition}
Generic circle patterns are rigid.

\begin{theorem}
\label{thm:circle-ring_deformation}
For any rigid orthogonal circle pattern solution of a Dirichlet or Neumann boundary value problem and sufficiently small $\epsilon$ there exists a ring pattern solving the same boundary value problem, parametrized by $q=1-\epsilon$ with the circle pattern corresponding to $q=1$. 
\end{theorem}
\begin{proof}
For small $\epsilon$ we should show the existence of functions $\beta_i(\epsilon)$ solving the system of equations
\begin{equation}
\label{eq:epsilon-family}
f_i(\beta, q):=\frac{\partial S_{sph}(\beta(\epsilon), q=1-\epsilon)}{\partial \beta_i}=0,
\end{equation}
for all rings (labeled by $i=1,\ldots,N$). 
This is a system of $N$ equations $f_i(\beta, q)=0$ for $N$ functions $\beta_j(\epsilon)$, satisfied at $\epsilon=0$ by the circle pattern radii $\beta(\epsilon=0)=\beta^0$.  
Since the circle pattern is rigid the matrix $\frac{\partial f_i}{\partial \beta_j}$ can be inverted at $\epsilon=0$, and by implicit function theorem there exists a solution for small $\epsilon$.  
Further, differentiation of \eqref{eq:epsilon-family} gives
$$
\sum_j \frac{\partial f_i}{\partial \beta_j}\frac{d\beta_j}{d\epsilon}-\frac{\partial f_i}{\partial q}=0,
$$
and for $\beta_i(\epsilon)$ we obtain a system of ordinary differential equations 
$$
\frac {d\beta_i}{d\epsilon}=\sum_j \left[ \frac{\partial^2 S_{sph}}{\partial\beta^2}\right]^{-1}_{i,j} \frac{\partial^2 S_{sph}(\beta, q)}{\partial q \partial \beta_j}, \quad q=1-\epsilon,
$$
with the initial condition $\beta(\epsilon=0)=\beta^0$.   
\end{proof}
This flow can be used for numerical computation of the ring patterns.

\section{Orthogonal ring patterns equation as a Laplace equation for Q4 integrable equation}
\label{sec:integrable}
 In this section we show that the basic equations (\ref{eq:spherical}, \ref{eq:hyperbolic}) for orthogonal ring patterns in the sphere and hyperbolic plane are integrable. They turned out to be a special case of the Q4 equation.

Quad equations on the square grid $x:\Z^2\to \C$ relate fields at the vertices of elementary squares of the grid. Integrable, i.e. multi-dimensionally  consistent, quad equations were classified in \cite{AdBSu03}, see also \cite{AdBSu09}. These equations are of the form
$$
Q(x_1,x_2,x_3,x_4; \alpha,\beta)=0,
$$
where $x_i$ are the fields at the vertices of a square and $\alpha, \beta$ are parameters associated to its edges, see Fig.\ref{fig:quad-equation}.

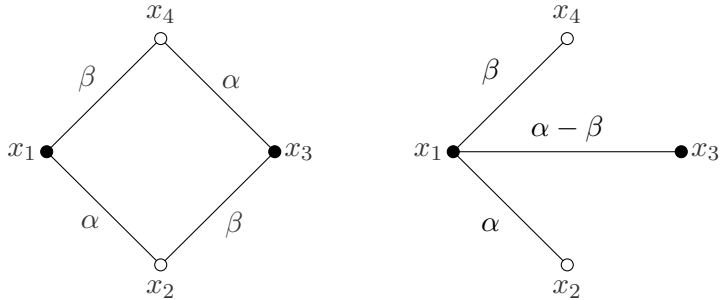
\begin{figure}[htbp] 
\begin{center}
    \begin{tikzpicture}[scale=1.5]
    \def\ptsize{1.5pt}
    \tikzset{label node/.style={inner sep=0, outer sep=.1cm, circle, fill=white, fill opacity=0.8}}
    \coordinate (x1) at (-1,0);
    \coordinate (x2) at (0,-1);
    \coordinate (x3) at (1,0);
    \coordinate (x4) at (0,1);
    \node[label node, anchor=east] at (x1) {$x_1$};
    \node[label node, anchor=north] at (x2) {$x_2$};
    \node[label node, anchor=west] at (x3) {$x_3$};
    \node[label node, anchor=south] at (x4) {$x_4$};
    \draw (x1) -- 
    node[label node, midway, anchor=north east] {$\alpha$} 
    (x2) -- 
    node[label node, midway, anchor=north west] {$\beta$} 
    (x3) -- 
    node[label node, midway, anchor=south west] {$\alpha$} 
    (x4) -- 
    node[label node, midway, anchor=south east] {$\beta$} 
    cycle;
    \foreach \complex in {x1, x3}
    {
      \draw [fill=black] (\complex) circle [radius=\ptsize];
    }
    \foreach \complex in {x2, x4}
    {
      \draw [fill=white] (\complex) circle [radius=\ptsize];
    }
  \end{tikzpicture}
  \hspace{1cm}
  \begin{tikzpicture}[scale=1.5]
    \def\ptsize{1.5pt}
    \tikzset{label node/.style={inner sep=0, outer sep=.1cm, circle, fill=white, fill opacity=0.8}}
    \coordinate (x1) at (-1,0);
    \coordinate (x2) at (0,-1);
    \coordinate (x3) at (1,0);
    \coordinate (x4) at (0,1);
    \node[label node, anchor=east] at (x1) {$x_1$};
    \node[label node, anchor=north] at (x2) {$x_2$};
    \node[label node, anchor=west] at (x3) {$x_3$};
    \node[label node, anchor=south] at (x4) {$x_4$};
    \draw (x1) -- node[midway, anchor=south] {$\alpha-\beta$} (x3);
    \draw (x1) -- node[midway, anchor=south east] {$\beta$} (x4);
    \draw (x1) -- node[midway, anchor=north east] {$\alpha$} (x2);
    \foreach \complex in {x1, x3}
    {
      \draw [fill=black] (\complex) circle [radius=\ptsize];
    }
    \foreach \complex in {x2, x4}
    {
      \draw [fill=white] (\complex) circle [radius=\ptsize];
    }

  \end{tikzpicture}
\end{center}
\caption{Fields at vertices and weights at the edges of an elementary square of a quad-equation. Three-leg form of a quad-equation}
\label{fig:quad-equation}
\end{figure}

The Q4 equation is the master equation in the list. All other integrable equations can be obtained from it by taking appropriate limits. It was found in \cite{Ad98} in the Weierstrass normalization of an elliptic curve, see \cite{AdSu04} for an elaborated treatment. In the Jacobi form this equation appeared in \cite{hietarinta_2005_cac}:
\begin{multline}
  \label{eq:Q4}
  \sn(\alpha;k)(x_1x_2+x_3x_4) \\
  - \sn(\beta;k)(x_1x_4+x_2x_3)
  - \sn(\alpha-\beta;k)(x_1x_3+x_2x_4) \\
  + \sn(\alpha-\beta;k)\sn(\alpha;k)\sn(\beta;k)(1+k^2x_1x_2x_3x_4)=0.
\end{multline}
Here $\sn(\alpha;k)$ is the Jacobi elliptic function of modulus $k$. The periods of the corresponding elliptic integrals are $4K$ and $4i\tilde{K}'$, and 
$$
\tilde{\tau}=i\frac{\tilde{K}'}{K}
$$ 
is the period of the elliptic curve.  Note that here we use tilde notations for the imaginary period $i\tilde{K}'$ and the period $\tilde\tau$.

Quad-equations imply Laplace type equations on a sublattice of $\Z^2$, which is in focus of our interest. 

Recall that 
$$
\sn(X;k)=\frac{\Theta_1(X)}{\sqrt{k}\Theta_4(X)},
$$
where we use an old-fashion normalization of the Jacobi theta functions, see \cite{Ak90}. It is related to a modern one, see \cite{BaEr55}, by an argument scaling:
$$
\Theta_i(X|\tilde{\tau})=\theta_i(\frac{X}{2K}|\tilde{\tau}), \quad i=1,2,3,4.
$$

As it was shown in \cite{bobenko_suris_2010} 
after the change of variables $x_i=\sn(X_i;k)$ the Q4 equation (\ref{eq:Q4}) can be presented in the following three-leg form:
$$
\dfrac{F(X_1,X_2,\alpha)}{F(X_1,X_4,\beta)}=F(X_1,X_3,\alpha-\beta)
$$
with 
\begin{multline} \label{eq:F_original}
F(X,Y,\alpha)=  
\dfrac{
\Theta_2(\frac{1}{2}(X+Y+\alpha)|\tilde{\tau})
\Theta_3(\frac{1}{2}(X+Y+\alpha)|\tilde{\tau})
}
{
\Theta_2(\frac{1}{2}(X+Y-\alpha)|\tilde{\tau})
\Theta_3(\frac{1}{2}(X+Y-\alpha)|\tilde{\tau})
}
\\
\cdot\dfrac{
\Theta_1(\frac{1}{2}(X-Y+\alpha)|\tilde{\tau})
\Theta_4(\frac{1}{2}(X-X+\alpha)|\tilde{\tau})
}
{
\Theta_1(\frac{1}{2}(X-Y-\alpha)|\tilde{\tau})
\Theta_4(\frac{1}{2}(X-Y-\alpha)|\tilde{\tau})
}.
\end{multline}

Three-leg form is closely related to the pluri-Lagrangian structure of the Q4 equation \cite{bobenko_suris_2010}. It turnes out that this structure is the one used for variational description of orthogonal ring patterns in Section~\ref{sec:variational}. 

Using the identities with the elliptic theta functions of the half-period
\begin{align*}
 2\Theta_1(v |\tilde{\tau})\Theta_4(v |\tilde{\tau})=\Theta_2(0 |\frac{\tilde{\tau}}{2})\Theta_1(v |\frac{\tilde{\tau}}{2}), \quad
2\Theta_2(v |\tilde{\tau})\Theta_3(v |\tilde{\tau})=\Theta_2(0 |\frac{\tilde{\tau}}{2})\Theta_2(v |\frac{\tilde{\tau}}{2}),
\end{align*}
 the formula for $F$ can be simplified:
\begin{align*}
F(X,Y,\alpha)=
\dfrac{
\Theta_2(\frac{1}{2}(X+Y+\alpha)|\frac{\tilde{\tau}}{2})
\Theta_1(\frac{1}{2}(X-Y+\alpha)|\frac{\tilde{\tau}}{2})
}
{
\Theta_2(\frac{1}{2}(X+Y-\alpha)|\frac{\tilde{\tau}}{2})
\Theta_1(\frac{1}{2}(X-Y-\alpha)|\frac{\tilde{\tau}}{2})
}.
\end{align*}
Introducing the notations
$$
K':=\frac{\tilde{K}'}{2}, \ \tau:=\frac{\tilde{\tau}}{2}
$$
for the torus with the half period and shifting the variables
$$
X\mapsto X+K,\ Y\mapsto Y+K
$$
we arrive at the following form
\begin{align} 
F(X,Y,\alpha)=
\dfrac{
\Theta_1(\frac{1}{2}(X+Y+\alpha)|\tau)
\Theta_1(\frac{1}{2}(X-Y+\alpha)|\tau)
}
{
\Theta_1(\frac{1}{2}(X+Y-\alpha)|\tau)
\Theta_1(\frac{1}{2}(X-Y-\alpha)|\tau)
}.
\end{align}

For $\alpha=iK'$ one obtains
\begin{align*}
\frac{\Theta_1(\frac{1}{2}(X\pm Y+iK')|\tau)}
{\Theta_1(\frac{1}{2}(X\pm Y-iK')|\tau)}=
i\sqrt{q}e^{-\frac{\pi i}{4K}(X\pm Y)}\sn(\frac{1}{2}(X\pm Y+iK'); q),
\end{align*}
and 
\begin{align*}
F(X,Y,iK')=-\frac{\sn(\frac{1}{2}(X+ Y+iK'); q)}{\sn(\frac{1}{2}(X- Y-iK'); q)} e^{-\frac{\pi i}{2K}X}.
\end{align*}
Similarly,
\begin{align*}
F(X,Y,-iK')=-\frac{\sn(\frac{1}{2}(X+ Y-iK'); q)}{\sn(\frac{1}{2}(X- Y+iK');q)} e^{\frac{\pi i}{2K}X}.
\end{align*}
Here the modulus $q$ corresponds to the torus with the periods $4K, 4iK'$.

\begin{figure}[htbp]
\begin{center}
    \begin{tikzpicture}[scale=2]
    \def\ptsize{1.125pt}
    \tikzset{label node/.style={inner sep=0, outer sep=.1cm, fill=white, fill opacity=0.8}}
    \coordinate (X) at (0,0);
    \coordinate (X1) at (1,1);
    \coordinate (X2) at (-1,1);
    \coordinate (X3) at (-1,-1);
    \coordinate (X4) at (1,-1);
    \coordinate (Y1) at (1,0);
    \coordinate (Y2) at (0,1);
    \coordinate (Y3) at (-1,0);
    \coordinate (Y4) at (0,-1);
    \foreach \complex in {X1, X2, X3, X4, Y1, Y2, Y3, Y4}
    {
      \draw (X) -- (\complex);
    }
    \path (X) -- node[pos=.6, anchor=west] {$\alpha-\beta$} (X1);
    \path (X) -- node[pos=.6, anchor=east] {$\beta-\alpha$} (X2);
    \path (X) -- node[pos=.6, anchor=east] {$\alpha-\beta$} (X3);
    \path (X) -- node[pos=.6, anchor=west] {$\beta-\alpha$} (X4);

    \path (X) -- node[pos=.6, anchor=south] {$\alpha$} (Y1);
    \path (X) -- node[pos=.6, anchor=east] {$\beta$} (Y2);
    \path (X) -- node[pos=.6, anchor=south] {$\alpha$} (Y3);
    \path (X) -- node[pos=.6, anchor=west] {$\beta$} (Y4);

    \node[label node, anchor=north west] at (X) {$X$};
    \node[label node, anchor=south west] at (X1) {$X_1$};
    \node[label node, anchor=south east] at (X2) {$X_2$};
    \node[label node, anchor=north east] at (X3) {$X_3$};
    \node[label node, anchor=north west] at (X4) {$X_4$};
    \node[label node, anchor=west] at (Y1) {$Y_1$};
    \node[label node, anchor=south] at (Y2) {$Y_2$};
    \node[label node, anchor=east] at (Y3) {$Y_3$};
    \node[label node, anchor=north] at (Y4) {$Y_4$};

    \foreach \complex in {X1, X2, X3, X4, X}
    {
      \draw [fill=black] (\complex) circle [radius=\ptsize];
    }

    \foreach \complex in {Y1, Y2, Y3, Y4}
    {
      \draw [fill=white] (\complex) circle [radius=\ptsize];
    }
    
  \end{tikzpicture}
\end{center}
\caption{From the Q4 equation to its Laplace type equation. The weights at the edges are indicated}
 \label{fig:Q4-star}
\end{figure}
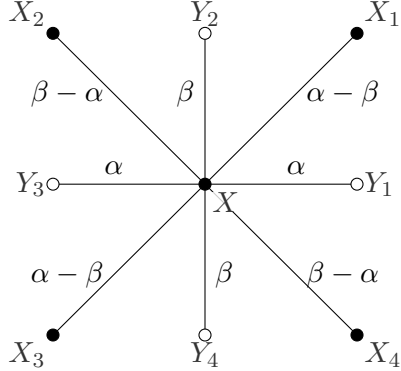

Let us show now how the orthogonal ring pattern equation (\ref{eq:spherical}) appears in this context. Consider a vertex of the $\Z^2$ lattice with its neighbors and next neighbors (see Fig.~\ref{fig:Q4-star}). The weights associated with horizontal and vertical edges are equal to $\alpha$ and $\beta$ respectively. Then the Q4 equations in the 3-leg form read as follows:
\begin{align*}
F(X,X_1,\alpha-\beta)=\frac{F(X,Y_1,\alpha)}{F(X,Y_2,\beta)}, \ 
F(X,X_2,\beta-\alpha)=\frac{F(X,Y_2,\beta)}{F(X,Y_3,\alpha)}\\
F(X,X_3,\alpha-\beta)=\frac{F(X,Y_3,\alpha)}{F(X,Y_4,\beta)}, \ 
F(X,X_4,\beta-\alpha)=\frac{F(X,Y_4,\beta)}{F(X,Y_1,\alpha)}.
\end{align*}
In the product the terms with $Y_i$ cancel, and we become the Laplace type equation in the multiplicative form
\begin{align}\label{eq:Laplace_F}
F(X,X_1,\alpha-\beta)
F(X,X_2,\beta-\alpha)
F(X,X_3,\alpha-\beta)
F(X,X_4,\beta-\alpha)=1.
\end{align}
For $\alpha-\beta=iK'$ the equation becomes
\begin{align}\label{eq:Laplace_preliminary}
&\frac{\sn(\frac{1}{2}(X+ X_1+iK'))}{\sn(\frac{1}{2}(X- X_1-iK'))}
\frac{\sn(\frac{1}{2}(X+ X_2-iK'))}{\sn(\frac{1}{2}(X- X_2+iK'))}\times \\
&\frac{\sn(\frac{1}{2}(X+ X_3+iK'))}{\sn(\frac{1}{2}(X- X_3-iK'))}
\frac{\sn(\frac{1}{2}(X+ X_4-iK'))}{\sn(\frac{1}{2}(X- X_4+iK'))}=1.\nonumber
\end{align}
This equation is on a square grid $\Z^2$ which is a sublattice of the original square grid of the Q4 equation. To identify this equation with (\ref{eq:spherical}) let us color, for example, vertical  lines of the grid of the Laplace system (\ref{eq:Laplace_preliminary}) alternatively black and white. The variables at black and white lines will be shifted in  opposite way: $Y\mapsto Y+iK'$ for black and $Y\mapsto Y-iK'$ for white vertices. If the vertices $X,X_2,X_4$ are on a black line, then the shifts in (\ref{eq:Laplace_preliminary}) are
$$
X\mapsto X+iK', X_{2,4}\mapsto X_{2,4}+iK', X_{1,3}\mapsto X_{1,3}-iK'.
$$
Using simple transformation identities for $\sn(X)$ function we obtain
\begin{align}\label{eq:Laplace_final}
&\frac{\sn(\frac{1}{2}(X+ X_1+iK');q)}{\sn(\frac{1}{2}(X- X_1+iK');q)}
\frac{\sn(\frac{1}{2}(X+ X_2+iK');q)}{\sn(\frac{1}{2}(X- X_2+iK');q)}\times \\
&\frac{\sn(\frac{1}{2}(X+ X_3+iK');q)}{\sn(\frac{1}{2}(X- X_3+iK');q)}
\frac{\sn(\frac{1}{2}(X+ X_4+iK');q)}{\sn(\frac{1}{2}(X- X_4+iK');q)}=1.\nonumber
\end{align}
This equation coincides with (\ref{eq:spherical}). One can easily verify that the same equation one obtains if the variables $X,X_2,X_4$ are on a white line. 

\begin{theorem} \label{th:LaplaceQ4_identification}
Equation (\ref{eq:spherical}), as well as (\ref{eq:hyperbolic}), describing the radii of the spherical and hyperbolic orthogonal ring patterns is a special case of the Laplace equation (\ref{eq:Laplace_F}) for Q4 discrete integrable equation on the square grid with parameter $\alpha-\beta$ on the edge equal to $iK'$.
\end{theorem}

Note that the modulus $k$ in (\ref{eq:Q4}) corresponds to the torus with the periods $4K, 4i\tilde{K}'$, which is the double cover of the torus with the periods $4K,4iK'$, corresponding to the modulus $q$ in (\ref{eq:Laplace_final}); $\tilde{K}'=2K'$.

\section{Smooth limit. Harmonic maps and elliptic sinh-Gordon equations}
\label{sec:smooth}

We describe the limit of small rings, smoothly varying in space, and demonstrate that in this limit orthogonal ring patterns converge to the corresponding harmonic maps.

\subsection*{In the sphere} 
First, let us consider the case of spherical ring patterns. If all rings are small then all $\beta$'s lie in the interval $[0,2K]$. Introducing $u:=K-\beta \in[-K,K]$ we obtain from \eqref{eq:sphere_ring_uniformization}
$$
\sin r=\cn \beta=q'\frac{\sn u}{\dn u},\quad \sin R=\dn \beta = \frac{q'}{\dn u}, \quad q'=\sqrt{1-q^2}.
$$
In the limit 
\begin{equation}
\label{eq:smooth_limit_q}
q'=\epsilon, \quad \epsilon\to 0, 
\end{equation}
we have small rings with the radii, see \eqref{eq:Jacobi_q=1}:
\begin{equation}
\label{eq:smooth_radii}
r\to q'\sinh u, \quad R\to q' \cosh u
\end{equation}
and the area $\pi(R^2-r^2)\to  q'^2$.

After simple transformations we obtain for the ingredients of \eqref{eq:spherical}:
\begin{eqnarray*}
\sn (\frac{\beta+\beta_k+iK'}{2})=(1+q)\sn (K -\frac{u_k+u}{2})\left( 
1+i\frac{q'^2}{1+q}\frac{\sn}{\cn\dn}(\frac{u_k+u}{2}) \right),\\
\sn (\frac{\beta-\beta_k+iK'}{2})=i \cn (\frac{u_k-u}{2}) \dn(\frac{u_k-u}{2})\left( 
1-i(1+q)\frac{\sn}{\cn\dn}(\frac{u_k-u}{2}) \right).
\end{eqnarray*}
In the smooth limit when the radii of the rings are assumed to depend on the ring coordinate smoothly we have:
\begin{equation}
\label{eq:smooth_limit_u}
u_k-u= \mathcal{O}(\epsilon), \quad \sum_{k=1}^4 (u_k-u)=\mathcal{O}(\epsilon^2),
\end{equation}
and equation \eqref{eq:spherical} becomes
\begin{equation}
\label{eq:limit_args}
\arg \prod_{k=1}^4\left(1+i \frac{q'^2}{1+q}\frac{\sn}{\cn\dn}(\frac{u_k+u}{2})\right)=
\arg \prod_{k=1}^4\left(1-i (1+q)\frac{\sn}{\cn\dn}(\frac{u_k-u}{2})\right).
\end{equation}
In the limit \eqref{eq:smooth_limit_q}, \eqref{eq:smooth_limit_u} this finally yields
\begin{equation}
\label{eq: Laplace u + sinh u=0}
\Delta u + \sinh 2u=0.
\end{equation}
Here we used 
$$
\frac{\sn}{\cn\dn}(x, q=1)=\frac{1}{2}\sinh 2x, \ \text{and} \ \sum_{k=1}^4\sinh(u_k-u)\to \sum_{k=1}^4(u_k-u)\to \Delta u,\  \epsilon\to 0,
$$
where $\Delta=\partial_x^2+\partial_y^2 $ is a properly normalized Laplace operator.

\subsection*{In the hyperbolic plane} 
For ring patterns in the hyperbolic plane the derivation of the smooth limit is similar.
 Introducing $u:=\gamma -K \in[-K,K]$ we obtain from \eqref{eq:hyperbolic_ring_uniformization_beta}, \eqref{eq:hyperbolic_ring_uniformization_gamma}
$$
\sinh r=i \dn (u+K+iK')=-q'\frac{\sn u}{\cn u},\quad \sinh R=i\cn (u+K+iK') = \frac{q'}{q\cn u}.
$$
In the limit \eqref{eq:smooth_limit_q} we again obtain small rings with the radii \eqref{eq:smooth_radii} and the area $q'^2$. Repeating the calculation in the spherical case we obtain from \eqref{eq:hyperbolic}: 
$$
\arg \prod_{k=1}^4 \left( 1-i \frac{q'^2}{1+q}\frac{\sn}{\cn\dn}(\frac{u_k+u}{2})\right)=
\arg \prod_{k=1}^4 \left( 1-i (1+q)\frac{\sn}{\cn\dn}(\frac{u_k-u}{2})\right).
$$
This equation differs from \eqref{eq:limit_args} just by one sign. In the smooth limit of small rings \eqref{eq:smooth_limit_q}, \eqref{eq:smooth_limit_u} we finally obtain:
\begin{equation}
\label{eq: Laplace u - sinh u=0}
\Delta u - \sinh 2u=0.
\end{equation}
We have shown that the difference equations \eqref{eq:spherical} and \eqref{eq:hyperbolic_ring_uniformization_beta} provide an approximation of the partial differential equations \eqref{eq: Laplace u + sinh u=0} and \eqref{eq: Laplace u - sinh u=0} in the usual sense of numerical mathematics, see for example \cite{Sa01}.  
\begin{theorem}
\label{the:smooth_limit}
Discrete equations \eqref{eq:spherical} and \eqref{eq:hyperbolic_ring_uniformization_beta} with $\beta=K-u$ and $\gamma=K-u$ in the limit \eqref{eq:smooth_limit_q}, \eqref{eq:smooth_limit_u} approximate the elliptic sinh-Gordon equations \eqref{eq: Laplace u + sinh u=0} and \eqref{eq: Laplace u - sinh u=0} 
respectively. If $u(x,y)$ is a solution of \eqref{eq: Laplace u + sinh u=0} or \eqref{eq: Laplace u - sinh u=0}, then the rings with the radii 
\begin{eqnarray*}
 r_{m,n}=\epsilon \sinh u(x,y), r_{m\pm 1,n}=\epsilon \sinh u(x\pm \epsilon,y), r_{m,n\pm 1}=\epsilon \sinh u(x,y\pm \epsilon), \\
 R_{m,n}=\epsilon \cosh u(x,y), R_{m\pm 1,n}=\epsilon\cosh u(x\pm \epsilon,y), R_{m,n\pm 1}=\epsilon \cosh u(x,y\pm \epsilon), 
\end{eqnarray*}
 satisfy the spherical or hyperbolic ring equations \eqref{eq:spherical} or \eqref{eq:hyperbolic_ring_uniformization_beta} respectively up to $\mathcal{O}(\epsilon)$. They form infinitesimal spherical or hyperbolic ring flowers (see Sect.~\ref{sec:ring_patterns}).
\end{theorem}

Note that both equations \eqref{eq: Laplace u + sinh u=0} and \eqref{eq: Laplace u - sinh u=0} are variational with the functionals
$$
S(u)=\int |\nabla u |^2\pm \cosh u,
$$
and, in agreement with the discrete case, the hyperbolic functional is convex.

\subsection*{Relation to harmonic maps and to constant mean curvature surfaces}

Elliptic sinh-Gordon equations \eqref{eq: Laplace u + sinh u=0} and \eqref{eq: Laplace u - sinh u=0}  describe cmc surfaces in $\R^3$ and $\R^{2,1}$ respectively. The function $u$ represents the conformal metric $e^{2u}$ of the surface.  The Gauss maps $N$ of these surfaces are harmonic map to the two-sphere $S^2$ and to the hyperbolic plane, see for example \cite{BHeSch19}. 

From considerations of the smooth limit in this section we conclude that orthogonal ring patters can be interpreted as integrable discretizations of harmonic maps to the two-sphere and to the hyperbolic plane. The corresponding discrete theory is developed in \cite{BHoSm24}, where discrete cmc surfaces are defined and constructed using orthogonal ring patterns. The surfaces are built by touching discs. In this context the radii of rings describe the conformal metric of the corresponding discrete cmc-surface. The Gauss map $N$ of a discrete cmc surface appears here in the form of so called two-spheres Koebe Q-net, see Fig.~\ref{fig:two-spheres_Koebe}, which we shortly describe below.

Let us decompose $G=G'\cup G''$ into two sublattices (recall that the vertices $(m,n)\in G$ satisfy $m+n=0 \ ({\rm mod} \ 2)$ 
$$
G'=\{ (m,n)\in G | m=0 \ ({\rm mod} \ 2) \}, \ G'=\{ (m,n)\in G | m=1 \ ({\rm mod} \ 2) \}.
$$
The lattices $G'$ and $G''$ are dual to each other; the faces $[(m,n), (m+2,n), (m+2,n+2), (m,n+2)] $ of $G'$ are dual to the vertices $(m+1,n+1)$ of $G''$ and vice versa. 

Recall that polyhedral surfaces in ${\mathbb R}^3$ with planar quadrilateral faces are called Q-nets \cite{BSu-DDG}.
Two Q-nets $p':G'\to {\mathbb R}^3$ and$p'':G''\to {\mathbb R}^3$ are called {\em a pair of dual two-spheres Koebe Q-nets} if they satisfy the following conditions:
\begin{enumerate}
\item $p'$ and $p''$ are combinatorially dual and vertex vectors of $p'$  build normals of the corresponding faces of $p''$ and vice versa.
\item The edges of face quadrilaterals $[p_{m,n}, p_{m+2,n}, p_{m+2,n+2}, p_{m,n+2}]$
of $p'$ and $p''$ alternatively touch two concentric spheres $S^2_\pm$:
the lines $(p_{m,n}, p_{m+2,n})$  touch $S^2_-$, and 
the lines $(p_{m,n}, p_{m,n+2})$  touch $S^2_+$.
\item The dual edges are orthogonal 
\begin{equation}
(p_{m,n}, p_{m+2,n})\perp (p_{m+1,n-1}, p_{m+1,n+1}),
\label{eq:orthogonality_dual}
\end{equation}
 and their touching points $l_{m+1,n}^+\in S^2_+$ and $l_{m+1,n}^-\in S^2_-$ are projected to the same point on the unit sphere.
\end{enumerate} 
An example of a pair of dual two-spheres Koebe Q-nets is shown in Fig.~\ref{fig:two-spheres_Koebe}.

Pairs of dual two-spheres Koebe Q-nets are in one to one correspondence to orthogonal ring patterns in $S^2$, see \cite{BHoSm24} for detail. Combined to one map $p:G\to \R^3, p'=p_{| G'}, p''=p_{| G''}$ they can be obtained by an appropriate scaling of the ring centers $k_{m,n}\in S^2$:
$$
p_{m,n}=\frac{1}{\sqrt{q}\cos r_{m,n}} k_{m,n}=\frac{\sqrt{q}}{\cos R_{m,n}}k_{m,n}.
$$ 
The concentric spheres $S_\pm^2$ have the radii
$$
r_+=\frac{1}{\sqrt{q}} \ \ \text{and}\ r_-=\sqrt{q}.
$$
Let us remark that the orthogonality condition (\ref{eq:orthogonality_dual}) was previously introduced in \cite{BSST_2020} in the context of discrete confocal quadrics, and later was elaborated in the theory of orthogonal binets \cite{AfTe24}.

\begin{figure}[htbp] 
\begin{center}
\includegraphics[width=0.8\textwidth]{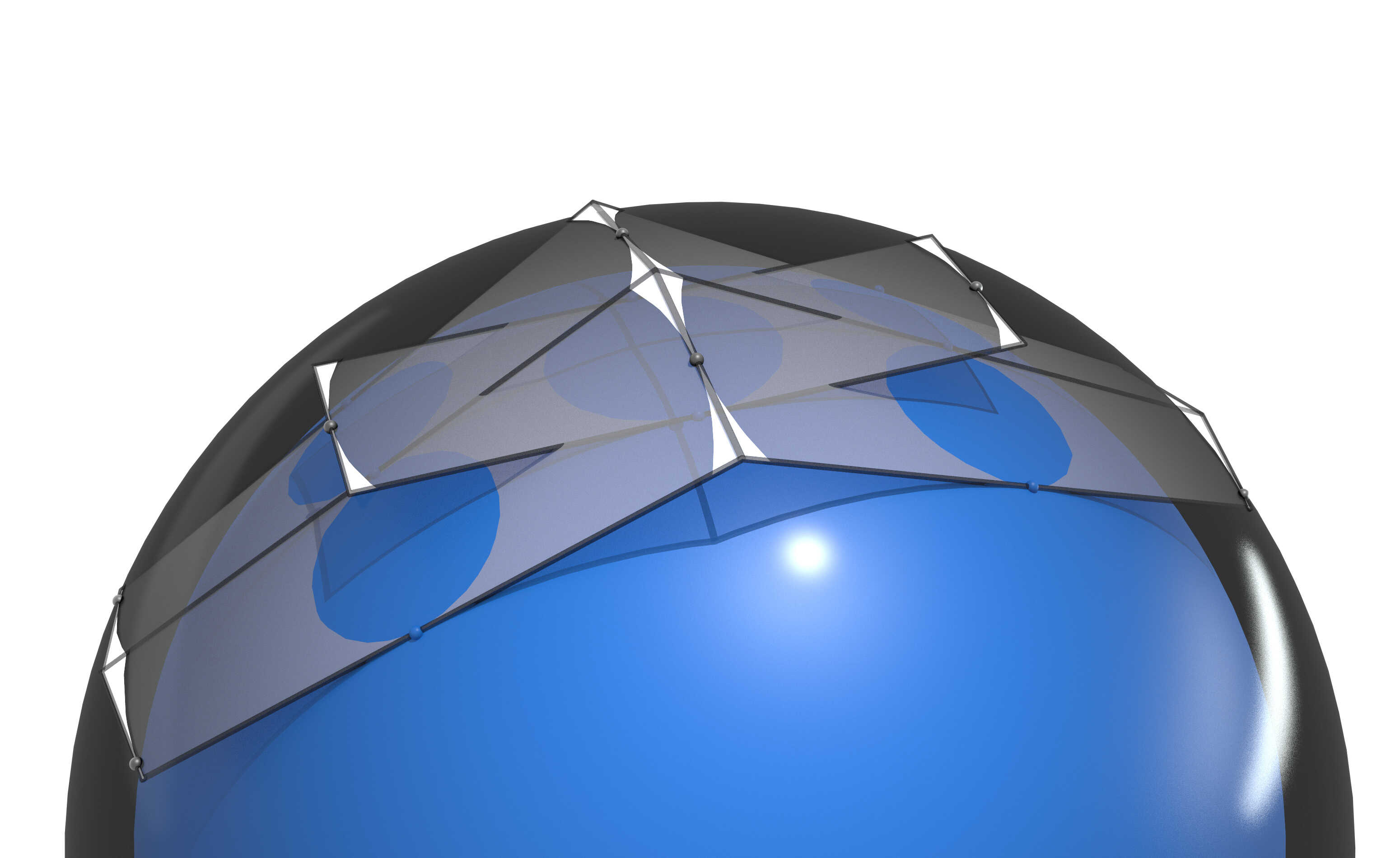}
\end{center}
\caption{A pair of dual two-spheres Koebe Q-nets obtained from an orthogonal ring pattern in $S^2$.}
\label{fig:two-spheres_Koebe}
\end{figure}

In the circle patterns case $q=1$ two spheres $S_\pm^2$ coincide with $S^2$, and one obtains classical Koebe polyhedra, all edges of which touch a sphere. The corresponding theory of discrete minimal surfaces constructed from orthogonal circle patters was developed in \cite{BHoSp06}. The construction steps are as follows: an orthogonal circle pattern in $S^2$ $\Rightarrow$  Koebe polyhedron $\Rightarrow$ discrete minimal surface. The convergence to the smooth minimal surfaces was proven using the convergence results for circle patterns \cite{Sch97, HeSch98, Bu08}. In the case of minimal surfaces the elliptic sinh-Gordon equation \eqref{eq: Laplace u + sinh u=0} degenerates to the Liouville equation $\Delta u=\frac{1}{2}e^{-u}$. The construction of discrete cmc surfaces in \cite{BHoSm24} is analogous: an orthogonal ring pattern in $S^2$ $\Rightarrow$  two-spheres Koebe polyhedron $\Rightarrow$ discrete cmc surface. Numerical experiments with ring patterns show good convergence to smooth surfaces and maps, and we expect for them a convergence behavior similar to the circle patterns and minimal surfaces. Note that orthogonal double circle patterns on the sphere have been applied in architectural geometry \cite{THDB18} for construction of doubly-curved building envelopes.

Cmc-surfaces in $\R^3$ are isometric to minimal surfaces in $S^3$. This fact, known as the Lawson correspondence \cite{La70}, can be lifted to a relation for the frames of the corresponding surfaces \cite{KGBPo97}. The Lawson correspondence is a powerful method to investigate embedded cmc-surfaces, by constructing the corresponding minimal surfaces in $S^3$, see for example \cite{GBKuSu03}. Analogously, cmc-surfaces in $\R^{2,1}$ are isometric to minimal surfaces in three-dimensional anti-de Sitter space $AdS_3$, which are popular models in the string theory.
First attempts to discretize the Lawson correspondence were carried out in \cite{BR_Lawson} without relation to ring patters. We plan to develop further the results of \cite{BHoSm24} and to find a discrete Lawson correspondence for ring patterns, which should be a geometric construction of discrete minimal surfaces in $S^3$ and $AdS_3$ from ring patterns in the sphere and hyperbolic plane. In particular, this will provide us an integrable model of discrete strings in $AdS_3$ with a convex variational principle. The latter is a well established tool to handle existence and uniqueness results.

\bibliographystyle{acm}
\bibliography{rings_noneucl}

\end{document}